\documentclass[a4paper,10pt]{amsart}
\RequirePackage{iftex}
\ifPDFTeX
  \RequirePackage[T1]{fontenc}
  \RequirePackage[utf8]{inputenc}
  \RequirePackage{lmodern}
\else
  \RequirePackage{lmodern}
  \RequirePackage[no-math]{fontspec}
\fi
\RequirePackage[a4paper,margin=1in]{geometry}
\RequirePackage{amsmath,amssymb,amsfonts,amsthm,mathtools}
\RequirePackage{microtype}
\ifPDFTeX
  \DisableLigatures{encoding=*,family=*}
\fi
\RequirePackage[none]{hyphenat}
\RequirePackage{enumitem}
\setlist[enumerate]{label=(\roman*),leftmargin=2em,itemsep=3pt,topsep=5pt}
\RequirePackage[svgnames]{xcolor}
\RequirePackage[colorlinks=true,citecolor=MediumBlue,linkcolor=Crimson,urlcolor=MediumBlue]{hyperref}
\numberwithin{equation}{section}
\theoremstyle{plain}
\newtheorem{theorem}{Theorem}[section]
\newtheorem{proposition}[theorem]{Proposition}
\newtheorem{lemma}[theorem]{Lemma}
\newtheorem{corollary}[theorem]{Corollary}
\theoremstyle{definition}
\newtheorem{definition}[theorem]{Definition}
\theoremstyle{remark}

\hypersetup{pdftitle={Mixed numerical criteria for Hessian equations on compact Kahler manifolds}}
\newcommand{\ddc}{\mathrm{i}\partial\bar\partial}
\newcommand{\R}{\mathbb R}
\newcommand{\Ccal}{\mathcal C}
\newcommand{\Kcal}{\mathcal K}
\newcommand{\Pcal}{\mathcal P}
\newcommand{\Pol}{\operatorname{Pol}}

\newcommand{\Gr}{\operatorname{Gr}}
\newcommand{\ac}{\mathrm{ac}}
\newcommand{\tr}{\operatorname{tr}}
\newcommand{\Int}{\operatorname{int}}

\newcommand{\Herm}{\operatorname{Herm}}

\newcommand{\supp}{\operatorname{supp}}
\title{ Nakai-Moishezon type criteria for Hessian type \\ equations on compact K\"ahler manifolds}
\author{Jixiang Fu}
\address{Shanghai Center for Mathematical Sciences,
Fudan University,
Shanghai 200433, China}
\email{majxfu@fudan.edu.cn}

\author{Xuan Li}
\address{Shanghai Institute for Mathematics and Interdisciplinary Sciences (SIMIS), Shanghai 200433, China}
\email{lixuan@simis.cn}

\author{Dekai Zhang}
\address{School of Mathematical Sciences, Key Laboratory of Mathematics
and Engineering Applications (Ministry of Education), Shanghai Key Laboratory
of PMMP, East China Normal University, Shanghai 200241, China}
\email{dkzhang@math.ecnu.edu.cn}

\author{Ziyi Zhang}
\address{School of Mathematical Sciences, Fudan University,
Shanghai 200433, China}
\email{21210180101@m.fudan.edu.cn}
\date{}
\begin{document}
\begin{abstract}
We establish  numerical criteria for polynomial positivity
of   real $(1,1)$-classes on compact K\"ahler manifolds for strictly  right-Noetherian polynomials.  As applications,
we obtain existence and uniqueness for the generalized Monge-Amp\`ere equations  with
a smooth zeroth coefficient, the complex Hessian and Hessian quotient
equations, and the critical LYZ equation. We also characterize
the cone of $k$-positive classes as a connected component
of a numerical positivity cone.
\end{abstract}
\maketitle
\begingroup
\makeatletter
\renewcommand{\tocsection}[3]{%
  \indentlabel{\@ifnotempty{#2}{%
    \makebox[2.5em][l]{\ignorespaces#1 #2.}}}#3}
\makeatother
\tableofcontents
\endgroup

\section{Introduction}\label{sec:introduction}

Yau's solution of the Calabi conjecture \cite{Yau} shows that
every smooth positive volume form with the prescribed total volume is
the volume form of a unique K\"ahler metric in a fixed K\"ahler class.
The resulting Calabi-Yau theorem is a fundamental existence theorem for the
complex Monge-Amp\`ere equation. The study of the $J$-equation, complex
Hessian equations, and Hessian quotient equations has led to existence
theorems under suitable positivity or subsolution conditions
\cite{SongWeinkove,FangLaiMa,HouMaWu,DinewKolodziej,Sun,Szekelyhidi}.
A central question is whether these conditions can be characterized by
intersection numbers. Demailly and P\u{a}un \cite{DP} answered this
question for the K\"ahler cone. Their theorem identifies it with a
connected component of the classes whose top self-intersection is
positive on every irreducible analytic subvariety.

For the $J$-equation, Lejmi and Sz\'ekelyhidi
\cite{LejmiSzekelyhidi} conjectured a numerical criterion for solvability.
Collins and Sz\'ekelyhidi \cite{CollinsSzekelyhidi} proved the toric case.
Chen \cite{Chen} established the criterion under a uniform numerical
condition on compact K\"ahler manifolds, and Song \cite{Song} treated
strict positivity without a uniform constant.
Sz\'ekelyhidi \cite{Szekelyhidi} also proposed a numerical criterion
for Hessian quotient equations in a $k$-positive class.
Related criteria for generalized Monge-Amp\`ere equations were proved
by Datar and Pingali \cite{DatarPingali} on projective manifolds and
by Fang and Ma \cite{FangMa} for equations with differential form
coefficients, under their respective hypotheses.

The Leung-Yau-Zaslow equation, which we call the LYZ equation, is of
particular interest here. It arises from the relation between special
Lagrangian submanifolds and Hermitian connections under the
Fourier-Mukai transform \cite{LYZ}. It is also called the deformed
Hermitian-Yang-Mills equation. Jacob and Yau \cite{JacobYau} developed
its analytic theory for holomorphic line bundles. Collins and Yau
\cite{CollinsYau} studied its variational structure and its relation
to stability in mirror symmetry.
On a compact K\"ahler manifold $(X,\chi)$ of complex dimension $n$,
the equation asks for a smooth closed real $(1,1)$-form
$\omega\in[\alpha]$ such that
\begin{equation}\label{eq:LYZintro}
 \sum_{j=1}^n\arctan\lambda_j(\chi^{-1}\omega)=\hat\theta.
\end{equation}
Here $\alpha$ is a fixed smooth closed real $(1,1)$-form, and
$\arctan$ takes values in $(-\pi/2,\pi/2)$.
The critical phase is $\hat\theta=(n-2)\pi/2$.

For the supercritical phase, Collins, Jacob, and Yau \cite{CJY}
obtained a priori estimates under a subsolution condition and derived
numerical obstructions to solvability. Chen \cite{Chen} related
solvability to uniform numerical positivity. Chu, Lee, and Takahashi
\cite{CLT} proved a numerical criterion along test families on compact
K\"ahler manifolds. Fu, Yau, and Zhang \cite{FYZFlow} introduced a
flow for the LYZ equation and proved convergence under their
subsolution and phase assumptions. At the critical phase, Fu, Yau,
and Zhang \cite{FYZ} proved an existence theorem under a subsolution
condition. Our application to the LYZ equation gives a numerical
criterion for the existence of such a subsolution.

Chen, Nie, and Xu \cite{CNX} recently proved numerical criteria for
a class of complex Hessian equations on projective manifolds.
Their results include uniform criteria for the complex Hessian and
Hessian quotient equations. The relevant polynomials are
right-Noetherian, in the sense introduced by Lin
\cite{LinConvexity}; see also Fang and Ma \cite{FangMaGarding}.
We prove criteria on compact K\"ahler manifolds using mixed
intersection numbers with one auxiliary K\"ahler form. In the
projective case, an ample integral class can be realize the numerical criteria in~\cite{CNX} by taking complete intersections. The need to specify
the component or to impose further numerical conditions is also
illustrated by Zhang's examples \cite{Zhang} for the LYZ equation.

Throughout the paper, Greek letters such as $\alpha$, $\chi$, and
$\kappa$ denote differential forms, and their cohomology classes
are written $[\alpha]$, $[\chi]$, and $[\kappa]$.
Products of forms denote wedge products. Unless stated otherwise,
$(X,\chi)$ is a compact connected K\"ahler manifold of complex
dimension $n$, and $\alpha$ is a smooth closed real $(1,1)$-form.

\subsection{Polynomial equations}\label{sec:statement}

Let $f$ be a monic polynomial of degree $1\le d\le n$, written as
\begin{equation}\label{eq:polynomial}
 f(t)=t^d-\sum_{i=0}^{d-1}\binom di c_i t^i,
 \qquad c_i\in\mathbb R.
\end{equation}
The definitions of right-Noetherianity and the associated cones
$\mathcal C_f(\chi)$ and $\mathcal D_f(\chi)$ are given in
Section~\ref{sec:cones}. In particular, strict right-Noetherianity
requires a strict gap only between the largest roots of $f$ and $f'$;
the later derivative roots may coincide.

We use the following mixed intersection inequalities:
\begin{equation}\label{eq:mixed}
 \int_V\kappa^{p-q}\chi^{n-d}
 \left(\alpha^{d-n+q}
 -\sum_{i=0}^{d-n+q-1}\binom{d-n+q}{i}
          c_{n-q+i}\alpha^i\chi^{d-n+q-i}\right)>0,
\end{equation}
where $V\subset X$ is irreducible of dimension $p$, including
$V=X$, and $n-d\le q\le\min\{p,n-1\}$. An empty sum is zero.
The K\"ahler form $\kappa$ is fixed independently of $V,p,q$.

\begin{theorem}\label{thm:fullPDE}
Let $c_1,\ldots,c_{n-1}$ be real constants and let
$c_0\in C^\infty(X,\mathbb R)$. Suppose that
\[
 f_x(t)=t^n-\sum_{i=1}^{n-1}\binom ni c_i t^i-c_0(x)
\]
is strictly right-Noetherian at every point of $X$, and assume
\begin{equation}\label{eq:fullidentity}
 \int_X\alpha^n
 =\sum_{i=1}^{n-1}\binom ni c_i\int_X\alpha^i\chi^{n-i}
       +\int_Xc_0\chi^n.
\end{equation}
Then the following conditions are equivalent.
\begin{enumerate}
\item There is a K\"ahler form $\kappa$ for which
\eqref{eq:mixed} holds with $d=n$.
\item There is a smooth closed form
$\omega\in[\alpha]\cap\mathcal D_{f_x}(\chi)$ satisfying
\begin{equation}\label{eq:fullPDE}
 \omega^n=\sum_{i=1}^{n-1}\binom ni c_i\omega^i\chi^{n-i}
                    +c_0(x)\chi^n.
\end{equation}
\end{enumerate}
Moreover, $\omega-c_{n-1}\chi$ is K\"ahler.
\end{theorem}

The inequalities in \textup{(i)} do not involve $c_0$.
This is a mixed numerical criterion on compact K\"ahler manifolds
for the equation in \cite[Theorem 1.1]{CNX}. We use the analytic
existence theorem of Lin \cite{Lin} after constructing a subsolution.
For polynomials of lower degree, we obtain the following result.

\begin{theorem}\label{thm:main}
Let $f$ be as in \eqref{eq:polynomial}, with constant coefficients,
and suppose that
\begin{equation}\label{eq:topzero}
 \int_X\alpha^d\chi^{n-d}
 =\sum_{i=0}^{d-1}\binom di c_i\int_X\alpha^i\chi^{n-i}.
\end{equation}
For $d\ge2$, assume that $f$ is strictly right-Noetherian.
For $d=1$, no further root condition is required.
The following conditions are equivalent.
\begin{enumerate}
\item There is a K\"ahler form $\kappa$ for which
\eqref{eq:mixed} holds.
\item For every sufficiently small $\varepsilon>0$, there is a
smooth closed form
$\omega_\varepsilon\in[\alpha]\cap\mathcal C_{f+\varepsilon}(\chi)$.
\end{enumerate}
If these conditions hold and $d\ge2$, then \textup{(ii)} holds whenever
$0<\varepsilon<-f(r(f'))$, where $r(f')$ is the largest real root
of $f'$. For $d=1$, it holds for every $\varepsilon>0$.
If these conditions hold, the left side of \eqref{eq:mixed} is
bounded below by $\delta_\kappa\int_V\kappa^{p-q}\chi^q$ for every
K\"ahler form $\kappa$, with one constant $\delta_\kappa>0$
independent of $V,p,q$.
\end{theorem}

Theorem~\ref{thm:main} is the mixed K\"ahler analogue of
\cite[Theorem 1.2]{CNX}. It includes $f(t)=t^k-a$, with $a>0$,
even though the later derivative roots coincide when $k\ge3$.
For a general polynomial of degree less than $n$, the theorem
constructs forms in the required cone. Solvability of the
corresponding equation requires a separate analytic existence theorem.

\subsection{The critical LYZ equation}

\begin{theorem}[Critical LYZ equation]\label{thm:critical}
Suppose that
\[
 \int_X(\alpha+\sqrt{-1}\chi)^n\in\mathbb R_{<0}.
\]
The following conditions are equivalent.
\begin{enumerate}
\item There is a K\"ahler form $\kappa$ such that
\begin{equation}\label{eq:criticalstrict}
 \int_V\kappa^{p-q}\operatorname{Im}
       (\alpha+\sqrt{-1}\chi)^q>0,
 \qquad 1\le q\le\min\{p,n-1\},
\end{equation}
for every irreducible analytic subvariety $V\subset X$ of
dimension $p$, including $V=X$.
\item The class $[\alpha]$ contains a smooth closed form $\xi$
satisfying
\begin{equation}\label{eq:criticalsub}
 \sum_{i\ne j}\arctan\lambda_i(\chi^{-1}\xi)
       >(n-3)\frac\pi2,\qquad 1\le j\le n.
\end{equation}
\item There is a smooth closed form $\omega\in[\alpha]$ satisfying
\begin{equation}\label{eq:criticalPDE}
 \sum_{j=1}^n\arctan\lambda_j(\chi^{-1}\omega)
       =(n-2)\frac\pi2.
\end{equation}
\end{enumerate}
If these conditions hold, then for every K\"ahler form $\kappa$
there is a constant $\delta_\kappa>0$ such that
\begin{equation}\label{eq:criticaluniform}
 \int_V\kappa^{p-q}\operatorname{Im}
       (\alpha+\sqrt{-1}\chi)^q
 \ge\delta_\kappa\int_V\kappa^{p-q}\chi^q
\end{equation}
for all $V,p,q$ as in \textup{(i)}.
\end{theorem}

Thus strict mixed inequalities give a uniform lower bound, without
assuming one in advance. The implication from \textup{(ii)} to
\textup{(iii)} follows from the critical existence theorem of
Fu, Yau, and Zhang \cite[Theorem 1.1]{FYZ}. The geometric part is
the construction of a strict subsolution from \textup{(i)}.
The critical equation corresponds to the polynomial
$n^{-1}\operatorname{Im}(t+\sqrt{-1})^n$, whose degree is $n-1$.
It therefore falls within Theorem~\ref{thm:main}.

\subsection{The Hessian and Hessian quotient equations}

A real $(1,1)$-form is called $k$-positive if its eigenvalues with
respect to $\chi$ belong to
$\Gamma_k=\{\lambda:\sigma_j(\lambda)>0,\ 1\le j\le k\}$.
The analytic existence theorems for the complex Hessian equation
in a $k$-positive class are due to Sun and Sz\'ekelyhidi
\cite{Sun,Szekelyhidi}. For Hessian quotient equations, we use
\cite[Corollary 3 and Proposition 22]{Szekelyhidi}.

\begin{theorem}[Hessian quotient equation]\label{thm:quotient}
Let $1\le\ell<k\le n$, suppose
$\int_X\alpha^\ell\chi^{n-\ell}\ne0$, and set
\[
 C=\frac{\int_X\alpha^k\chi^{n-k}}
         {\int_X\alpha^\ell\chi^{n-\ell}}>0.
\]
The following conditions are equivalent.
\begin{enumerate}
\item There is a K\"ahler form $\kappa$ such that
\begin{equation}\label{eq:quotmixed}
 \int_V\kappa^{p-q}
 \left(\frac{k!}{(k-n+q)!}\alpha^{k-n+q}\chi^{n-k}
 -C\frac{\ell!}{(\ell-n+q)!}
       \alpha^{\ell-n+q}\chi^{n-\ell}\right)>0
\end{equation}
for every irreducible analytic $V\subset X$ of dimension $p$,
including $X$, and $n-k\le q\le\min\{p,n-1\}$.
The second term is omitted when $q<n-\ell$.
\item There is a smooth closed $k$-positive form $\omega\in[\alpha]$
satisfying
\begin{equation}\label{eq:quotPDE}
 \omega^k\chi^{n-k}=C\omega^\ell\chi^{n-\ell}.
\end{equation}
\end{enumerate}
If $[\alpha]$ already contains a smooth closed $k$-positive form,
these conditions are also equivalent to
\begin{equation}\label{eq:quotendpoint}
 \int_V\left(\frac{k!}{(k-n+p)!}\alpha^{k-n+p}\chi^{n-k}
 -C\frac{\ell!}{(\ell-n+p)!}
       \alpha^{\ell-n+p}\chi^{n-\ell}\right)>0
\end{equation}
for every irreducible analytic $V\subset X$ of dimension
$n-\ell\le p<n$.
\end{theorem}

The last assertion is the numerical criterion proposed by
Sz\'ekelyhidi \cite{Szekelyhidi}. In particular, the mixed inequalities imply that $[\alpha]$
contains a smooth closed $k$-positive form. This is not assumed
in the equivalence of \textup{(i)} and \textup{(ii)}.

\begin{theorem}[Complex Hessian equation]\label{thm:hessian}
Let $1\le k\le n$ and suppose $\int_X\alpha^k\chi^{n-k}>0$.
The following conditions are equivalent.
\begin{enumerate}
\item There is a K\"ahler form $\kappa$ such that
\begin{equation}\label{eq:puremixed}
 \int_V\kappa^{p-q}\alpha^{k-n+q}\chi^{n-k}>0,
 \qquad n-k\le q\le\min\{p,n-1\},
\end{equation}
for every irreducible analytic $V\subset X$ of dimension $p$,
including $X$.
\item For every smooth real function $F$ satisfying
$\int_Xe^F\chi^n=\int_X\alpha^k\chi^{n-k}$, there is a smooth
closed $k$-positive form $\omega\in[\alpha]$ satisfying
\begin{equation}\label{eq:purePDE}
 \omega^k\wedge\chi^{n-k}=e^F\chi^n.
\end{equation}
\end{enumerate}
\end{theorem}

We also prove that the cone of $k$-positive classes is the connected
component containing $[\chi]$ of the classes satisfying
\[
 \int_V\alpha^{k-n+p}\chi^{n-k}>0,
 \qquad n-k+1\le p=\dim V\le n.
\]
See Theorem~\ref{thm:component}. This gives the corresponding
criterion along the classes $[\alpha]+t[\kappa]$, $t\ge0$,
and proves the path criterion formulated by Murakami
\cite[Conjecture 1.5(1)]{Murakami}. When $k=n$, the component
description is the theorem of Demailly and P\u{a}un \cite{DP}.

\subsection{Outline of the proof}

The main step is a geometric existence theorem for a
right-Noetherian polynomial $f$ of degree $n$ with positive top integral.
We show that the mixed inequalities imply the existence of a
smooth closed form in the associated cone. The case of zero top
integral follows by replacing $f$ with $f+\varepsilon$.
This changes only the top integral, and the strict first root gap
preserves right-Noetherianity for small $\varepsilon>0$.

We start from the continuity method. At each continuity endpoint,  we use the mass concentration  motivated by the argument in Tosatti's remarkble work~\cite{Tosatti}.  
Tosatti's key observation
is that every K\"ahler class $[\kappa]$ contains a closed positive current
supported on a closed set $\Sigma$ of Lebesgue measure zero
\cite[Proposition~2.1]{Tosatti}. Since this set has Lebesgue measure zero,  all mass of any positive
current on $\Sigma$ comes from its singular part.
 We apply a distance estimate for the polynomial cone  to get a lower
bound for the singular part of the limiting current.
Then by a Lamari type duality, we obtain a current $T\in [\alpha]$ with $(T-\delta\kappa)_{\ac}\in\overline{\Ccal_f(\chi)}\;{a.e.}$ for a constant $\delta>0$.

For a polynomial $f$ of degree $d<n$, we apply the mass concentration argument to
$p(t)=(t+L)^{n-d}f(t)$ with $L$ sufficiently large.
The cone inclusion lemma in  \cite{CNX} allow us to
pass between the two polynomials. We keep one finite value of $L$
when constructing the current in a limiting class. This retains
a smooth lower bound needed for regularization.  The key ingredient is the current constructed in
Proposition~\ref{prop:endpoint}. 
Mass concentration provides that  there exists a current $T\in[\alpha]$ with $(T-\delta\kappa)_{\ac}\in \overline{C_{f}(\chi)}$ and $T\geq -C\chi$ needed for gluing. The constant $\delta$ will be chosen independent of $L$. Moreover, this step requires no uniform numerical lower bound in~\cite[Theorem~1.2, (3)]{CNX}, this is why we do not need the lower bound here.

We then regularize local potentials and glue them to smooth
potentials near analytic subsets. This uses the regularization
theory of Demailly and Richberg
\cite{DemaillyRegularization,Richberg,BlockiKolodziej}, together
with induction on the dimension. The local extension argument
is related to those of Demailly-P\u{a}un. Collins-Tosatti and Chen
\cite{DP,CollinsTosatti, Chen}. A continuity argument along
$[\alpha]+t[\kappa]$ completes the geometric existence theorem.
Finally, intersection inequalities on resolutions of subvarieties
give the component description for $k$-positive classes and reduce
the Hessian quotient criterion to the proper subvariety inequalities.

Section~\ref{sec:cones} contains the polynomial definitions and
the analytic existence theorem. Sections~\ref{sec:currents}
and~\ref{sec:concentration} concern convex duality and mass
concentration. Sections~\ref{sec:gluing}--\ref{sec:global}
prove local extension and geometric existence. The remaining
sections treat the Hessian equations, their numerical components,
and the critical LYZ equation.

\section{Preliminary}\label{sec:cones} Let $(X,\chi)$ be a compact K\"ahler manifold of complex dimension $n$.
Let $\Lambda_{\mathbb R}^{1,1}(X)$ denote the space of real
$(1,1)$-forms on $X$. For $\xi\in\Lambda_{\mathbb R}^{1,1}(X)$,
we write $\xi\ge0$ if $\xi$ is semipositive. Write $\sigma_j$ for the elementary symmetric polynomials,
with $\sigma_0=1$. 
For $1\le d\le n$, denote
\[
 \Gamma_d(\chi)
 :=
 \left\{
 \xi\in\Lambda_{\mathbb R}^{1,1}(X):
 \sigma_j\bigl(\lambda(\chi^{-1}\xi)\bigr)>0,
 \quad 1\le j\le d
 \right\},
\]
where $\lambda(\chi^{-1}\xi)$ denotes the vector of eigenvalues of $\xi$
with respect to $\chi$. In particular, $\Gamma_n(\chi)$ is the cone of
positive real $(1,1)$-forms.   The cone of classes containing a smooth closed
k-positive form is \[
\mathcal K_{k,\chi}
=\bigl\{[\alpha]\in H^{1,1}(X,\mathbb R):
\text{there exists a smooth closed }
\omega\in[\alpha]\text{ with }\omega\in\Gamma_k(\chi)\bigr\}.
\]

More generally, given a positive $(1,1)$-form
$\chi$ and a unitarily invariant set $K\subset\Herm(n)$, we write
$\xi\in K(\chi)$ if the matrix of $\xi$ in a $\chi$-unitary frame belongs
to $K$ at every point of $X$. This condition is independent of
the choice of frame.

\subsection{Polynomial notation and mixed intersection numbers}
\mbox{}\par
\vspace{0.5\baselineskip}
\noindent
We first recall the following root conditions from  \cite{Lin}. 
\begin{definition}
Suppose that $f$ is a  real univariate polynomial of degree \(d\). We associate with \(f\) the following root sequence. For \(0\leq j\leq d-1\), denote the largest real root of \(f^{(j)}\), when it exists, by $r(f^{j})$.  We say \(f\) is right-Noetherian if  
\[
r(f)\geq r(f')\geq\cdots\geq r\bigl(f^{(d-1)}\bigr).
\]
 We say \(f\) is strictly right-Noetherian if
\[
r(f)>r(f')\geq\cdots\geq r\bigl(f^{(d-1)}\bigr).
\
\]
\end{definition}
The polynomial $t^d-a$ with $a>0$  is strictly right-Noetherian, although
its later derivative roots coincide.
Set $r_1=r(f')$. Since $f$ is strictly increasing to the right
of $r_1$, we have
\begin{equation}\label{eq:firstgapvalue}
 r(f)>r_1\quad\Longleftrightarrow\quad f(r_1)<0.
\end{equation}
Consequently, if the first gap is strict, then $f+\varepsilon$
is strictly right-Noetherian for
$0<\varepsilon<-f(r_1)$.
If $f$ is merely right-Noetherian, then $f(r_1)\le0$,
so $f-\varepsilon$ has a strict first gap for any
$\varepsilon>0$.

 Let \(f\) be as in \(\eqref{eq:polynomial}\). We define \(F=\operatorname{Pol}_n(f)\) as the unique symmetric polynomial of total degree \(d\) in \(n\) variables that is affine in each variable and satisfies 
\[
F(t,\ldots,t)=f(t).
\]

For each nonzero partial derivative $P$ of
$F$,  including $P=F$, we denote by  $\Upsilon_P$ the connected
component of
\[
 \{\lambda\in\mathbb R^n:P(\lambda)>0\}
\]
that contains $R\mathbf1$ for all sufficiently large $R$, where
$\mathbf1=(1,\ldots,1)$. If $P$ is a positive constant,  set
$\Upsilon_P=\mathbb R^n$.
Since \(P\) is affine in each variable,
\begin{equation}\label{eq:diagonalcomponent}
	P(\lambda+t\mathbf1)
	=
	\sum_{I\subset{1,\ldots,n}}
	t^{|I|}\partial_IP(\lambda).
\end{equation}
If $P$ and all its nonzero partial derivatives are positive
at $\lambda$, the expression above is positive for $t\ge0$.
Therefore, for large $R$, the segment joining
	$\lambda+t\mathbf1$ to $R\mathbf1$ lies in ${P>0}$.
	Hence $\lambda\in\Upsilon_P$. 
    
 Now suppose that \(f\) is right-Noetherian. Its polarization \(F\) has the Fang-Ma Gårding structure by \cite[Theorem 8.2]{FangMaGarding}.
 In particular, for every
nonzero partial derivative $P$ of $F$ and every $j$, we have
\(\Upsilon_P\subset\Upsilon_{\partial_jP}.\) Together with the preceding argument, this gives
\[
\lambda\in\Upsilon_P
\quad\Longleftrightarrow\quad
\partial_IP(\lambda)>0
\quad\text{whenever }\partial_IP\not\equiv0.
\]
 We write
\[
 \mathcal C_f=\Upsilon_F,
 \qquad
 \mathcal D_f=\bigcap_{j=1}^n\Upsilon_{\partial_jF}.
\]
The set $\mathcal D_f$ is the first derivative cone
$\Upsilon_F^1$ of \cite[Definition 2.14]{CNX}. When $d=n$,
it agrees with Lin's $\Upsilon_1$-cone
\cite[Definition 2.6]{LinConvexity}. The convention introduced above defines \(\mathcal C_f(\chi)\) and \(\mathcal D_f(\chi)\) for every positive real \((1,1)\)-form \(\chi\), which consists of the real \((1,1)\)-forms \(\xi\) whose eigenvalue vectors \(\lambda(\chi^{-1}\xi)\) lie in \(\mathcal C_f\) and \(\mathcal D_f\), respectively, at every point.

 For real $(1,1)$-forms
$\eta$ and $\xi$, define $\Phi_{n-d}^f(\eta,\xi)=\xi^{n-d}$ and
\begin{equation}\label{eq:Phi}
 \Phi_{n-d+j}^f(\eta,\xi)
 =\xi^{n-d}\left(\eta^j-
       \sum_{i=0}^{j-1}\binom ji c_{d-j+i}\eta^i\xi^{j-i}\right),
 \qquad 1\le j\le d.
\end{equation}
For closed forms, the same notation denotes the corresponding
expression in their cohomology classes. We suppress the superscript
$f$ when the polynomial is fixed. Thus \eqref{eq:mixed} is
\[
 \int_V[\kappa]^{p-q}\Phi_q^f([\alpha],[\chi])>0,
 \qquad n-d\le q\le\min\{p,n-1\}.
\]
These inequalities also hold for $V=X$. The uniform version
of Theorem~\ref{thm:main} reads
\begin{equation}\label{eq:alluniform}
 \int_V[\kappa]^{p-q}\Phi_q^f([\alpha],[\chi])
 \ge\delta_\kappa\int_V\kappa^{p-q}\chi^q.
\end{equation}
In particular, $\kappa=\chi$ and $q=p<n$ give
\begin{equation}\label{eq:properuniform}
 \int_V\Phi_p^f([\alpha],[\chi])
       \ge\delta_\chi\int_V\chi^p,
 \qquad n-d\le p<n.
\end{equation}

 For the Hessian
equations, we use the abbreviations
\[
 V_j=\int_X\alpha^j\chi^{n-j},\qquad C=V_k/V_\ell,
\]
and
\begin{equation}\label{eq:psi}
 \Psi_q=
 \begin{cases}
 [\alpha]^{k-n+q}[\chi]^{n-k},&n-k\le q<n-\ell,\\[3pt]
 [\alpha]^{k-n+q}[\chi]^{n-k}
 -C\theta_q[\alpha]^{\ell-n+q}[\chi]^{n-\ell},&n-\ell\le q\le n,
 \end{cases}
\end{equation}
where
\[
 \theta_q=\frac{\ell!(k-n+q)!}{k!(\ell-n+q)!}
 =\prod_{j=0}^{k-\ell-1}\frac{k-n+q-j}{k-j},
 \qquad n-\ell\le q\le n.
\]
Multiplication by $k!/(k-n+q)!$ identifies the positivity of
$\int_V[\kappa]^{p-q}\Psi_q$ with \eqref{eq:quotmixed}.

\subsection{Binomial expansion  and formal derivative}
Let $\gamma>0$. We choose a $\gamma$-unitary frame
in which $\xi$ is diagonal, with eigenvalues $\lambda_1,\ldots,\lambda_n$.
Then
\begin{equation}\label{eq:wedgepolar}
 \frac{\xi^j\gamma^{n-j}}{\gamma^n}
 =\frac{\sigma_j(\lambda)}{\binom nj},
 \qquad
 \frac{\Phi_n^f(\xi,\gamma)}{\gamma^n}=\Pol_n(f)(\lambda).
\end{equation}

For a real $(1,1)$-form $\sigma$, the binomial theorem gives
\begin{equation}\label{eq:formtaylor}
 \Phi_n^f(\xi+t\sigma,\gamma)
 =\Phi_n^f(\xi,\gamma)+\sum_{j=1}^d\binom dj t^j\sigma^j\Phi_{n-j}^f(\xi,\gamma).
\end{equation}
More generally, we have
\begin{equation}\label{eq:derivativetaylor}
 \Phi_q^f(\xi+t\sigma,\gamma)
 =\sum_{j=0}^{d-n+q}\binom{d-n+q}{j}
 t^j\sigma^j\Phi_{q-j}^f(\xi,\gamma),
 \qquad n-d\le q\le n.
\end{equation}

For $c\in\mathbb R$, let $\widetilde f(t)=f(t+c)$.
Polarization commutes with translation, so
\begin{equation}\label{eq:translationforms}
	\Phi_q^{\widetilde f}(\xi-c\gamma,\gamma)
	=
	\Phi_q^f(\xi,\gamma),
	\qquad n-d\le q\le n.
\end{equation}

 We next express these forms in terms of the partial derivatives
	of $F$.
	At a fixed point, choose a $\gamma$-unitary coframe diagonalizing
	$\xi$ and write
	\[
	\gamma=\sum_{j=1}^n\beta_j,
	\qquad
	\xi=\sum_{j=1}^n\lambda_j\beta_j,
	\qquad
	\beta_j=\sqrt{-1}\,dz^j\wedge d\bar z^j.
	\]
	For $I\subset\{1,\ldots,n\}$, write
	\[
	\beta_I=\bigwedge_{j\in I}\beta_j,
	\qquad
	\partial_I=\prod_{j\in I}\frac{\partial}{\partial\lambda_j},
	\qquad
	I^c=\{1,\ldots,n\}\setminus I.
	\]
	Expanding the wedge products, we have (see \cite[Lemma 2.22]{CNX}) 
	\begin{equation}\label{binomial}
	\Phi_q^f(\xi,\gamma)
	=
	\frac{n!(d-n+q)!}{d!}
	\sum_{|I|=q}
	\partial_{I^c}F(\lambda)\,\beta_I,
	\qquad n-d\le q\le n.
	\end{equation}
	For right-Noetherian $f$, it follows that
	$\xi\in\mathcal C_f(\gamma)$ makes these forms strictly strongly
	positive for $n-d\le q\le n$.
	If $\xi\in\mathcal D_f(\gamma)$, the same conclusion holds for
	$n-d\le q<n$.

We now consider restrictions to complex subspaces.
Let $E\subset T_x^{1,0}X$ be a complex $p$-plane,
where $n-d<p\le n$, and define
\begin{equation}\label{eq:inducedpolynomial}
	g=\frac{(d-n+p)!}{d!}f^{(n-p)}.
\end{equation}
This polynomial is monic of degree $d-n+p$.
Its forms in ambient dimension $p$ satisfy
\begin{equation}\label{eq:inducedforms}
	\Phi_q^g(\xi|_E,\gamma|_E)
	=
	\Phi_q^f(\xi,\gamma)|_E,
	\qquad n-d\le q\le p.
\end{equation}
Suppose that $f$ is right-Noetherian and
	$\xi\in\mathcal C_f(\gamma)$.
	The forms $\Phi_q^f(\xi,\gamma)$ are strictly strongly positive,
	as are their restrictions to $E$ for $n-d\le q\le p$.
	By \eqref{eq:inducedforms}, these restrictions are
	$\Phi_q^g(\xi|_E,\gamma|_E)$, 
	and $\Pol_p(g)$ and all its nonzero partial derivatives
	are positive at the eigenvalues of $\xi|_E$ relative to $\gamma|_E$.
	It therefore shows
	$$\xi|_E\in\mathcal C_g(\gamma|_E).$$

\subsection{Geometry of the polynomial cones}
 In this subsection we collect the properties of the cone associated with  right-Noetherian polynomial. 
\begin{lemma}\label{lem:geometry}
Let $g(t)=t^d-\sum_{i<d}\binom di c_it^i$ be right-Noetherian,
with $1\le d\le n$. Then $\mathcal C_g$ is open, convex,
unitarily invariant, and stable under addition of semipositive
matrices. Its closure has interior $\mathcal C_g$, and
$\overline{\mathcal C_g}+\Gamma_n\subset\mathcal C_g$.
All nonzero partial derivatives of $\Pol_n(g)$ are positive on
$\mathcal C_g$. In particular, $\Phi_q^g(\xi,\chi)$ is
strictly strongly positive for $\xi\in\mathcal C_g(\chi)$
and $n-d\le q\le n$.
Denote $c_g:=c_{d-1}$, we have
\begin{equation}\label{eq:shiftGamma}
 \Ccal_g(\chi)-c_g\chi\subset\Gamma_d(\chi).
\end{equation}
If $d=n$, any element of the closure of the translated set
is semipositive. For $p_L(t)=(t+L)^{n-d}g(t)$ and any
compact $B\Subset\mathcal C_g$, all sufficiently large $L>0$
satisfy
\begin{equation}\label{eq:finiteLift}
 B\subset\mathcal C_{p_L}\subset\mathcal C_g.
\end{equation}
\end{lemma}

\begin{proof} 
Let \(F=\Pol_n(g)\). By
\cite[Section~9.2]{FangMaGarding},
\(F\) has the  G{\aa}rding structure. Hence
\(\mathcal C_g\) is open and unitarily invariant, all nonzero
partial derivative of \(G\) are positive on \(\mathcal C_g\), and
\[
 \mathcal C_g+\Gamma_n\subset\mathcal C_g.
\]
 Applying \cite[Theorem~10.10]{FangMaGarding} to
\(g(t+c_g)\) and \(t^d\)  also gives
\[
 \mathcal C_g-c_gI\subset\Gamma_d.
\]

The convexity of $\mathcal{C}_g$ is given in \cite[Theorem~3.1]{LinConvexity} when $f$ is strictly  right-Noetherian with full degree. Now 
 set \(m=n-d\) and Write
\[
F_L=\operatorname{Pol}_n(p_L).
\]
If \(|I|=r\leq d\), then, uniformly on compact subsets of \(\mathcal C_g\),
\[
 \partial_I F_L
 =
  \partial_I F \  L^m
 +O(L^{m-1}).
\]
If \(d<r\leq n\), then
\[
 \partial_I F_L  
 =
 \frac{m!r!}{n!(r-d)!} L^{n-r}
 +O(L^{n-r-1}).
\]
Hence the component criterion gives
\(B\subset\mathcal C_{p_L}\) for every
\(B\Subset\mathcal C_g\) and all sufficiently large \(L\), while
\(\mathcal C_{p_L}\subset\mathcal C_g\) follows from
\cite[Lemma~3.1]{CNX}. This proves \eqref{eq:finiteLift}.

Given \(A_0,A_1\in\mathcal C_g\), choose \(L\) so that both matrices lie in
\(\mathcal C_{p_L}\). For sufficiently
small \(\varepsilon>0\), we may assume  
\(A_0,A_1\in\mathcal C_{p_L-\varepsilon}\). This polynomial is strictly
right-Noetherian, so its cone is convex by
\cite[Theorem~3.1]{LinConvexity}. Moreover,
\[
 \mathcal C_{p_L-\varepsilon}
 \subset\mathcal C_{p_L}
 \subset\mathcal C_g
\]
by \cite[Theorem~10.10]{FangMaGarding}. Thus \(\mathcal C_g\) is
convex, and
\(\operatorname{int}\overline{\mathcal C_g}=\mathcal C_g\).
\end{proof}

The following lemma gives an explicit pointwise distance to the boundary of the cone.
\begin{lemma}\label{lem:depth}
If $2\le d\le n$, $\xi\in\Ccal_g(\chi)$ and $\chi, \sigma>0$ be $(1,1)$-forms, then
\begin{equation}\label{eq:depth}
 \xi-\frac{1}{d}\frac{\Phi_n^g(\xi,\chi)}{\Phi_{n-1}^g(\xi,\chi)\wedge \sigma}\sigma
       \in\overline{\Ccal_g(\chi)}.
\end{equation}
\end{lemma}
\begin{proof}
The argument is pointwise. Set
\[
  F(\gamma):=\frac{\Phi_n^g(\gamma,\chi)}{\chi^n},
  \qquad
  t_*:=\sup\{t\geq0:\xi-t\sigma\in\Ccal_g(\chi)\}.
\]
Since $\Ccal_g(\chi)$ is open and $\xi\in\Ccal_g(\chi)$, we have
$t_*>0$. Moreover,  by convexity we have
\[
  \{t\geq0:\xi-t\sigma\in\Ccal_g(\chi)\}=[0,t_*).
\]

Let $c_g$ be the translation constant in \eqref{eq:shiftGamma}.
For $0\leq t<t_*$, that inclusion implies
\[
  0<\operatorname{tr}_\chi(\xi-t\sigma-c_g\chi)
   =\operatorname{tr}_\chi \xi-t\operatorname{tr}_\chi \sigma-nc_g.
\]
Since $\operatorname{tr}_\chi \sigma>0$, it follows that
\[
  0<t_*
  \leq
  \frac{\operatorname{tr}_\chi \xi-nc_g}
       {\operatorname{tr}_\chi \sigma}
  <\infty.
\]
By the definition of $t_*$,
 $\xi_*:=\xi-t_*\sigma\in\partial\Ccal_g(\chi)$.
As $\Ccal_g(\chi)$ is a connected component of $\{F>0\}$,
we have
\(
  F(\xi_*)=0.
\)
Indeed, if $F(\xi_*)>0$, a sufficiently small ball around $\xi_*$
would lie in $\{F>0\}$, intersect $\Ccal_g(\chi)$, and hence
belong to the same component, contradicting
$\xi_*\in\partial\Ccal_g(\chi)$.

Set $\varphi(t)=F(\xi-t\sigma)$ for $0\leq t<t_*$. 
By expansion \eqref{eq:derivativetaylor}, for \(0\le t<t_*\), we have
\[
  \varphi'(t)
  =-\frac{d\Phi_{n-1}^g(\xi-t\sigma,\chi)\wedge \sigma}{\chi^n}<0,
  \qquad
  \varphi''(t)
  =\frac{d(d-1)\Phi_{n-2}^g(\xi-t\sigma,\chi)\wedge \sigma^2}{\chi^n}>0.
\]
The second inequality implies $\varphi'(t)\geq\varphi'(0)$.
Since $\varphi(t_*)=0$, we obtain
\[
  \varphi(0)
  =-\int_0^{t_*}\varphi'(t)\,dt
  \leq -t_*\varphi'(0).
\]
Consequently,
\[
  0<\frac{\Phi_n^g(\xi,\chi)}{d\Phi_{n-1}^g(\xi,\chi)\wedge \sigma}
  =\frac{\varphi(0)}{-\varphi'(0)}
  \leq t_*.
\]
Since $\xi-t\sigma\in\overline{\Ccal_g(\chi)}$ for
$0\leq t\leq t_*$, this proves \eqref{eq:depth}.
\end{proof}

\subsection{Existence from subsolutions in the full degree case}
 We recall the existence theorem under subsolution conditions proved  in  \cite{Lin}
 \begin{theorem} [{\cite[Theorem~1.4]{Lin}}]
Let
$ 
 p_x(t)=t^n-\sum_{i=0}^{n-1}\binom ni c_i(x)t^i,$ 
where \(c_i\) is constant for \(i\ge1\) and \(c_0\) is smooth.
Assume that \(p_x\) is strictly right-Noetherian for every \(x\in X\).
If \([\alpha]\) contains a smooth closed form in
\(\mathcal D_{p_x}(\chi)\) and
$ \int_X\Phi_n^{p_x}(\alpha,\chi)=0, $ 
then there exists a unique smooth form  $ 
 \omega\in[\alpha]\cap\mathcal D_{p_x}(\chi)
$ 
satisfying $\Phi_n^{p_x}(\omega,\chi)=0. $ 
\end{theorem}

We may normalize \(c_{n-1}=0\). Indeed, set
\(c=c_{n-1}\), \(\eta=\omega-c\chi\), and write
\[
 p_x(y+c)
 =
 y^n-\sum_{j=0}^{n-1}\binom nj\widetilde c_j(x)y^j.
\]
Then \(\widetilde c_{n-1}=0\), while
\(\widetilde c_j\) is constant for \(j\ge1\). The translated equation is
\[
 \eta^n
 =
 \sum_{j=0}^{n-2}\binom nj\widetilde c_j(x)
 \eta^j\wedge\chi^{n-j}.
\]
 By
\eqref{eq:translationforms}, the translation preserves the  gaps of root and the normalization condition. Since \(c\)
is constant, \(\eta\) is closed and represents
\([\alpha]-c[\chi]\).
Finally, we have $ 
 \bigl(p_x(y+c)\bigr)^{(n-1)}=n!y. $ 
Then the derivative cone condition implies
\(\eta>0\), hence
\[
 \omega-c\chi>0.
\]

\section{Convex duality}\label{sec:currents}

Let $\mathcal P\subset\Herm(n)$ be the cone of semipositive
matrices. In this section, we fix a nonempty closed convex set
$K\subset\mathcal P$ such that
\begin{equation}\label{eq:Kconditions}
 UKU^*=K\quad(U\in U(n)),\qquad K+\mathcal P\subset K.
\end{equation}
We denote its realization in a $\chi$-unitary frame by $K(\chi)$ and  use the measure $\mu=\chi^n$ throughout. 

For a  positive $1,1$-current $T$, we write its Lebesgue decomposition as  (see~\cite{Bou02})
\[
 T=T_{\ac}+T_{\mathrm s},\qquad
 T_{\ac}\ll\mu,\qquad T_{\mathrm s}\perp\mu.
\]
Here $T_{\ac}$ and $T_{\mathrm{sing}}$ are positive currents
whose coefficient measures are, respectively, absolutely
continuous and singular with respect to the Lebesgue measure. Besides, if \(T\) is closed, then \(T_{\ac}^n\), which is defined a.e., is integrable and hence defines a distribution.  We now consider the following family
\begin{equation}\label{eq:currentfamily}
 \mathcal A=\{T\ge0:T_{\ac}(x)\in K(\chi)(x)
                                  \text{ for a.e. }x\}.
\end{equation}

\begin{lemma}\label{lem:weakclosed}
The family $\mathcal A$ is convex and weakly closed in the space of $(1,1)$-currents. If
$T\in\mathcal A$ and $P\ge0$ is a current, then
$T+P\in\mathcal A$. Moreover, $T \in \mathcal A$ implies $T_{ac}\in \mathcal A$.
\end{lemma}
\begin{proof}
We obtain convexity and stability under addition from
\[
 ((1-t)T_0+tT_1)_{\ac}
      =(1-t)(T_0)_{\ac}+t(T_1)_{\ac},
 \qquad (T+P)_{\ac}=T_{\ac}+P_{\ac}.
\]

To prove the weak closedness, let $T_j\in\mathcal A$ converge weakly
to $T$. Since each $T_j$ is positive, we have $T\ge0$.
Fix a relatively compact coordinate domain and a smooth \(\chi\)-unitary frame. In this frame, write
\[
dM_j=A_j\,d\mu+dM_{j,\mathrm s},
\qquad A_j(y)\in K\ \text{a.e.},\qquad dM_{j,\mathrm s}\ge0.
\]
The fixed frame identifies weak convergence of currents with weak convergence of the corresponding matrix-valued measures. Hence \(M_j\rightharpoonup M\), where \(M\) represents \(T\)

Let $\rho_\varepsilon\ge0$ be a smooth radial approximate
identity in the coordinate variables. We take $\varepsilon$
smaller than the distance to the outer boundary. Put
\begin{align*}
 Z_\varepsilon(x)&=\int\rho_\varepsilon(x-y)\,d\mu(y),\\
 A_{j,\varepsilon}(x)
 &=\frac{1}{Z_\varepsilon(x)}
             \int\rho_\varepsilon(x-y)\,dM_j(y).
\end{align*}
We have $Z_\varepsilon(x)>0$. The Lebesgue decomposition gives
\begin{equation}\label{eq:barycentre}
 A_{j,\varepsilon}(x)
 =\frac{\int\rho_\varepsilon(x-y)A_j(y)\,d\mu(y)}
        {Z_\varepsilon(x)}
  +\frac{\int\rho_\varepsilon(x-y)\,dM_{j,\mathrm s}(y)}
        {Z_\varepsilon(x)}.
\end{equation}
The first term is an average of matrices in $K$ and belongs to
$K$ by closed convexity. The second term is semipositive.
Consequently, \eqref{eq:Kconditions} implies
\[
 A_{j,\varepsilon}(x)\in K.
\]

We first fix $\varepsilon>0$. Since the convolution kernel is a
smooth compactly supported test function, weak convergence gives
\[
 A_{j,\varepsilon}(x)\longrightarrow
 A_\varepsilon(x)
 :=\frac{\int\rho_\varepsilon(x-y)\,dM(y)}{Z_\varepsilon(x)}.
\]
As $K$ is closed, we obtain $A_\varepsilon(x)\in K$. 
Now we let $\varepsilon\downarrow0$. Write $d\mu=w\,dx$,
where $w$ is smooth and positive and 
$dM=A\,d\mu+dM_{\mathrm s}$. At almost every point,  the Lebesgue differentiation theorem gives  
\begin{align*}
 Z_\varepsilon(x)=(\rho_\varepsilon*w)(x)&\longrightarrow w(x),\\
 \int\rho_\varepsilon(x-y)A(y)\,d\mu(y)
        &\longrightarrow A(x)w(x),\\
 \int\rho_\varepsilon(x-y)\,dM_{\mathrm s}(y)
        &\longrightarrow0.
\end{align*} 
It then follows that
\[
 A_\varepsilon(x)\longrightarrow A(x)\quad\text{a.e.}
\]
Since $K$ is closed, we obtain $A(x)\in K$ a.e.
A finite coordinate cover then gives $T\in\mathcal A$.
The cone of positive currents is metrizable in the weak topology \cite[Chapter III, Section 1.D]{DemaillyBook}. Therefore, the preceding sequential argument shows that \(\mathcal A\) is weakly closed. 
\end{proof}

\subsection{A duality criterion}

Lamari's criterion characterizes classes containing positive
currents by pairing against smooth strictly positive
$\ddc$-closed $(n-1,n-1)$-forms \cite[Lemme 3.3]{Lamari}.
We use the separation principle, motivated by 
\cite[Proposition 2.3]{Tosatti}, with an additional condition on
the absolutely continuous part of the current.

\begin{proposition}[Convex duality]\label{prop:duality}
Let $n\ge2$, let $K$ satisfy \eqref{eq:Kconditions}, and let
$\psi$ be a smooth closed real $(1,1)$-form. The following are
equivalent
\begin{enumerate}
\item There is a closed positive current $R\in[\psi]$ with
$R_{\ac}\in K(\chi)$ almost anywhere.
\item For any smooth strictly positive $(n-1,n-1)$-form
$\Gamma$ with $\ddc\Gamma=0$, there is a positive current
$\xi_\Gamma$ such that
\[
 (\xi_\Gamma)_{\ac}\in K(\chi)\quad\text{a.e.},\qquad
 \int_X(\xi_\Gamma)_{\ac}\wedge\Gamma
       \le\int_X\psi\wedge\Gamma.
\]
\end{enumerate}
\end{proposition}

\begin{proof}
We follow the separation argument in
\cite[Proposition 2.3]{Tosatti}. By averaging a matrix of $K$
over the unitary group, we obtain $aI\in K$ for some $a\ge0$.
Thus $a\chi\in\mathcal A$, so this family is nonempty.

We put
\[
 \mathcal F=\{\psi+\ddc u:u\text{ is a real distribution}\}.
\]
We first show that $\mathcal F$ is weakly closed. We normalize
$\int_Xu_j\chi^n=0$. If $\ddc u_j\to U$ in distributions,
we use the scalar Green operator to write
\[
 u_j=G_\chi\bigl(\tr_\chi\ddc u_j\bigr)
       \longrightarrow G_\chi(\tr_\chi U)=u.
\]
Here we choose the Laplacian normalization compatible with the
trace. Continuity of $G_\chi$ on distributions gives
$\ddc u=U$. Hence $\mathcal F$ is closed. The K\"ahler
$\partial\bar\partial$ lemma identifies $\mathcal F$ with the
closed real $(1,1)$-currents in the class of $\psi$.

any $S\in\mathcal F$ has the same mass:
\[
 \int_XS\chi^{n-1}=\int_X\psi\chi^{n-1}=:m_\psi.
\]
We claim that $\mathcal A-\mathcal F$ is weakly closed.
Let $W$ lie in its closure and choose the neighborhood
\[
 \mathcal U_W=\left\{U:
 \int_XU\chi^{n-1}<\int_XW\chi^{n-1}+1\right\}.
\]
If $U=T-S\in\mathcal U_W$ with $T\in\mathcal A$ and
$S\in\mathcal F$, then
\[
 \int_XT\chi^{n-1}
 <\int_XW\chi^{n-1}+1+m_\psi=:M.
\]
This mass bound and Lemma~\ref{lem:weakclosed} imply that the set
\[
 \mathcal A_M=\left\{T\in\mathcal A:
                       \int_XT\chi^{n-1}\le M\right\}
\]
is weakly compact; see \cite[Chapter III, Proposition 1.23]{DemaillyBook}.
Since $\mathcal F$ is closed, $\mathcal A_M-\mathcal F$ is closed.
It contains $\mathcal U_W\cap(\mathcal A-\mathcal F)$, and hence
its closure contains $W$. Thus $W\in\mathcal A-\mathcal F$.

Assume (ii) and suppose $\mathcal A\cap\mathcal F=\varnothing$.
We separate the closed convex set $\mathcal A-\mathcal F$ from
zero. We obtain a smooth real test form $\sigma$ and $a_*>0$ with
\begin{equation}\label{eq:separation}
 \int_X(T-S)\wedge\sigma\ge a_*
       \qquad(T\in\mathcal A,\ S\in\mathcal F).
\end{equation}
We may replace $S$ by $S+s\ddc v$ for any $s\in\R$.
Thus
\[
 \int_X\ddc v\wedge\sigma=0\quad\text{for any }v,
 \qquad\text{so}\qquad \ddc\sigma=0.
\]
We may also replace $T$ by $T+tP$ for any positive current
$P$ and any $t\ge0$. Dividing \eqref{eq:separation} by $t$
and letting $t\to\infty$, we obtain
\[
 \int_XP\wedge\sigma\ge0\quad(P\ge0).
\]
Hence $\sigma$ is semipositive.

For a small $\varepsilon>0$, put
$\Gamma=\sigma+\varepsilon\chi^{n-1}$. Then
\[
 \Gamma>0,\qquad \ddc\Gamma=0,
\]
and, for any $T\in\mathcal A$,
\begin{align*}
 \int_X(T-\psi)\wedge\Gamma
 &\ge a_*+\varepsilon\left(\int_XT\chi^{n-1}-m_\psi\right)\\
 &\ge a_*-\varepsilon m_\psi>0.
\end{align*}
By (ii), the absolutely continuous current
$T=(\xi_\Gamma)_{\ac}$ belongs to $\mathcal A$ and satisfies
$\int_X(T-\psi)\wedge\Gamma\le0$, a contradiction. Hence
$\mathcal A\cap\mathcal F\ne\varnothing$, proving (i).

Conversely, let $R$ satisfy (i) and take $\xi_\Gamma=R$.
Since $R_{\mathrm s}\ge0$ and $\ddc\Gamma=0$, we have
\[
 \int_XR_{\ac}\wedge\Gamma
 \le\int_XR\wedge\Gamma
 =\int_X\psi\wedge\Gamma.
\]
This proves (ii).
\end{proof}

We will also use the connection between the absolutely continuous
constraint and local smoothing. This identifies our currents with
the current cone used in \cite[Definition~7.2]{CNX}.

\begin{lemma}\label{lem:current-cone-mollification}
Let $K$ and $\mathcal A$ be as in Lemma~\ref{lem:weakclosed}.
A positive $(1,1)$-current $T$ belongs to $\mathcal A$
if and only if, for every coordinate ball $U$, every constant
positive form $h_0\le\chi$ on $U$, and every nonnegative
smooth mollifier $\rho_\varepsilon$ of integral one,
\[
 (T*\rho_\varepsilon)(x)\in K(h_0)
 \qquad\text{whenever } \text{dist}(x,U)>\varepsilon.
\]
Convolution is taken coefficientwise in the chosen coordinates.
\end{lemma}

\begin{proof}
In local coordinates, write the Lebesgue decomposition of
the coefficient matrix measure as
\[
 dM=A(y)\,dy+dM_{\mathrm s},
 \qquad A\ge0,\quad dM_{\mathrm s}\ge0,
\]
and set $M_\varepsilon=\rho_\varepsilon*M$.

Suppose $T\in\mathcal A$.
For $0<h_0\le\chi$, the min--max principle, unitary invariance,
and the property of $K$ implies that if 
 $A(y)\in K(\chi(y))$, then
 $A(y)\in K(h_0)$. Consequently,
\[
 M_\varepsilon(x)
 =\int\rho_\varepsilon(x-y)A(y)\,dy
  +\int\rho_\varepsilon(x-y)\,dM_{\mathrm s}(y)
 \in K(h_0).
\]
Indeed, the first term is an average in the closed convex
set $K(h_0)$, and the second is semipositive.

Conversely, fix a nonnegative smooth compactly supported
kernel $\rho$ of integral one and set
$\rho_\varepsilon(z)=\varepsilon^{-2n}\rho(z/\varepsilon)$.
Differentiation of measures gives
\[
 \rho_\varepsilon*(A\,dy)\to A,\qquad
 \rho_\varepsilon*M_{\mathrm s}\to0
 \quad\text{a.e.}
\]
Fix a point $x$ at which both limits hold, so that
$M_\varepsilon(x)\to A(x)$, and let $0<e<1$.
The constant form $h_e=(1-e)\chi(x)$ satisfies $h_e\le\chi$
on a sufficiently small coordinate ball around $x$.
Thus, for sufficiently small $\varepsilon$,
\[
\frac{\chi(x)^{-1/2}M_\varepsilon(x)\chi(x)^{-1/2}}{1-e}
 \in K.
\]
Letting first $\varepsilon\to 0$ and then $e\to 0$,
closedness gives $A(x)\in K(\chi(x))$.
Hence $T\in\mathcal A$.
\end{proof}

\section{Mass concentration}\label{sec:concentration}

Motivated by McCleerey's attempts to find plurisupported currents,
Tosatti observed that every Kähler class contains a closed positive
current supported on a closed set of Lebesgue measure zero (for more
concrete examples, see \cite{McCleerey}). Together with the Lebesgue decomposition of positive \((1,1)\)-currents, this  nullset supported  current is sufficient for the mass concentration and cone duality arguments.

\begin{lemma}[{\cite[Proposition~2.1]{Tosatti}}]
\label{lem:nullset}
For any K\"ahler form $\kappa$, there exist a closed positive
current $S\in[\kappa]$, a closed Lebesgue null set $\Sigma$, and
smooth K\"ahler forms $S_j\in[\kappa]$ such that
\[
 \supp S\subset\Sigma,
 \qquad
 S_j\rightharpoonup S.
\]
\end{lemma}
The current \(S\) is locally given by continuous potentials modeled on
\(\max\{\log|g|,0\}\), where \(g\) is a local holomorphic function.  Tosatti glues finitely
many of these local models to construct \(S\). The resulting current is
supported on the creases of the potentials and on the interfaces where
the models are glued. These sets lie in a closed Lebesgue-null set
\(\Sigma\).  The regularization theorem of
B{\l}ocki-Ko{\l}odziej
\cite{BlockiKolodziej} then gives K\"ahler forms
$S_j\in[\kappa]$ converging weakly to $S$.

\subsection{Mass concentration in full degree}

\begin{proposition}\label{prop:fullconcentration}
Let $p$ be a constant monic right-Noetherian polynomial of degree
$n$, written as $g(t)=t^n-\sum_{i<n}\binom ni c_it^i$. Suppose $[\alpha]$ contains a smooth closed form
$\omega_\alpha\in\Ccal_p(\chi)$. Let $c=c_{n-1}$, and let $\kappa$ be a K\"ahler form. Let
\begin{equation}\label{eq:concentrationratio}
 V=\int_X\Phi_n^p([\alpha],[\chi]),\qquad
 D=n\int_X\Phi_{n-1}^p([\alpha],[\chi])[\kappa],\qquad \rho=V/D.
\end{equation}
There is a closed current $Q$ with
\begin{equation}\label{eq:fullconcentratedcurrent}
 [Q]=\alpha-\rho[\kappa],\qquad Q\ge c\chi,
 \qquad Q_{\ac}\in\overline{\Ccal_p(\chi)}\quad\text{a.e.}
\end{equation}
Its singular part is positive.
\end{proposition}

\begin{proof}
 Since $\alpha$ contains a smooth closed form
$\omega_\alpha\in\Ccal_p(\chi)$, we have $V,D>0$. By
Lemma~\ref{lem:geometry}, the set
\[
 K(\chi)=\overline{\Ccal_p}-cI
\]
is closed, convex, unitarily invariant, and contained in the
semipositive matrices. We will apply Proposition~\ref{prop:duality}
to this set and the class $\alpha-c\chi-\rho[\kappa]$.

Fix a smooth strictly positive $\ddc$-closed form $\Gamma$ of
bidegree $(n-1,n-1)$ and put
\[
 M=\int_X\kappa\wedge\Gamma>0.
\]
Choose $S,\Sigma,S_j$ as in Lemma~\ref{lem:nullset}. We have
\[
 \int_XS_j\wedge\Gamma=\int_XS\wedge\Gamma=M,
\]
because all these currents represent $[\kappa]$ and
$\ddc\Gamma=0$. Choose decreasing open neighborhoods $U_j$ of
$\Sigma$ with $\overline{U_{j+1}}\subset U_j$ and
$\bigcap_jU_j=\Sigma$. Choose smooth functions $0\le\zeta_j\le1$
supported in $U_j$ and equal to one on $U_{j+1}$.

For each fixed $j$, weak convergence of the approximations gives
$\int_X\zeta_jS_l\wedge\Gamma\to M$ as $l\to\infty$.
We choose a subsequence, still denoted by $S_j$, for which
\begin{equation}\label{eq:nullmass}
 \int_{U_j}S_j\wedge\Gamma
 \ge\int_X\zeta_jS_j\wedge\Gamma\ge M-j^{-1}.
\end{equation}

We solve the smooth auxiliary equations
\begin{equation}\label{eq:auxiliaryPDE}
 \Phi_n^p(\omega_j,\chi)=\frac VM S_j\wedge\Gamma,
                     \qquad [\omega_j]=\alpha.
\end{equation}
To solve \eqref{eq:auxiliaryPDE}, we set
$q_j=(V/M)S_j\wedge\Gamma/\chi^n>0$ and $r_1=r(p')$.
Right-Noetherianity gives $p(r_1)\le0$, so
\[
 (p-q_j(x))(r_1)<0,
 \qquad (p-q_j(x))^{(a)}=p^{(a)}\quad(a\ge1).
\]
Thus $p-q_j(x)$ is strictly right-Noetherian at any point.
Its derivative cone is independent of $x$ and $j$. The form $\omega_\alpha$ is therefore a subsolution for any such equation. Also,
\[
 \int_Xq_j\chi^n=\frac VM\int_XS_j\wedge\Gamma=V.
\]
After the normalization in Section~\ref{sec:cones}, Lin's theorem
\cite[Theorems 1.4 and 2.2]{Lin} gives the solution $\omega_j$.
Its positive polynomial value and its derivative inequalities,
together with \eqref{eq:diagonalcomponent}, show
$\omega_j\in\Ccal_p(\chi)$.

By Lemma~\ref{lem:geometry}, $\psi_j:=\omega_j-c\chi>0$. We define the positive function
\[
 \tau_j=\frac{\Phi_n^p(\omega_j,\chi)}
                  {n\Phi_{n-1}^p(\omega_j,\chi)\wedge S_j}.
\]
The directional distance estimate (see Lemma~\ref{lem:depth}) and Lemma~\ref{lem:geometry} give
\[
 \psi_j-\tau_jS_j\in \overline{C_p}(\chi)-c\chi\subset\{P\ge0\}.
\]
Consequently, $\psi_j\ge\tau_jS_j$. For any measurable set $U$,
write $q=S_j\wedge\Gamma/\chi^n$ and
$P=n\Phi_{n-1}^p(\omega_j,\chi)\wedge S_j/\chi^n$. Since $\omega_j$ is the solution of equation~\eqref{eq:auxiliaryPDE},  we have
\begin{align}
 \int_U\psi_j\wedge\Gamma&\geq\int_U \tau_j S_j\wedge \Gamma
 \ge\frac VM\int_U\frac{q^2}{P}\chi^n\notag\\
 &\ge\frac VM\frac{(\int_Uq\chi^n)^2}{\int_UP\chi^n}
 \ge\frac{V}{MD}\left(\int_US_j\wedge\Gamma\right)^2.
                                                   \label{eq:nulltrace}
\end{align}
The third inequality follows from the Cauchy-Schwarz inequality.
The last inequality follows from the identity
\[
 \int_XP\chi^n
 =n\int_X\Phi_{n-1}^p(\omega_j,\chi)\wedge S_j
 =n\int_X\Phi_{n-1}^p(\alpha,\chi)[\kappa]=D.
\]
In particular, $D$ is independent of $j$ and $\Gamma$.

By \eqref{eq:nullmass} and \eqref{eq:nulltrace},
\begin{equation}\label{eq:concentratedmass}
 \liminf_{j\to\infty}\int_{U_j}\psi_j\wedge\Gamma\ge\rho M.
\end{equation}
The positive currents $\psi_j$ all represent $\alpha-c[\chi]$, so their
$\chi$-masses are bounded. We pass to a subsequence with
$\psi_j\rightharpoonup\xi_\Gamma$. Since $\psi_j\in K(\chi)$,
Lemma~\ref{lem:weakclosed} gives
\[
 \xi_\Gamma\ge0,\qquad
 (\xi_\Gamma)_{\ac}\in K(\chi)\quad\text{a.e.},
 \qquad [\xi_\Gamma]=\alpha-c[\chi].
\]

We now use fixed cutoffs to pass to the limit in
\eqref{eq:concentratedmass}. For $j\ge k+1$, the cutoff
$\zeta_k$ equals one on $U_j$. Hence
\[
 \int_X\zeta_k\xi_\Gamma\wedge\Gamma
 =\lim_j\int_X\zeta_k\psi_j\wedge\Gamma\ge\rho M.
\]
As $k\to\infty$, the cutoffs tend pointwise to $1_\Sigma$.
Dominated convergence for the positive measure
$\xi_\Gamma\wedge\Gamma$ gives
\[
 \int_\Sigma\xi_\Gamma\wedge\Gamma\ge\rho M.
\]
Since the positive measure $(\xi_\Gamma)_{\ac}\wedge\Gamma$ is absolutely continuous with
respect to Lebesgue, and $\Sigma$ has zero Lebesgue measure,
\[
 \int_\Sigma(\xi_\Gamma)_{\ac}\wedge\Gamma=0,
 \qquad
 \int_\Sigma(\xi_\Gamma)_{\text{sing}}\wedge\Gamma
 =\int_\Sigma\xi_\Gamma\wedge\Gamma\ge\rho M.
\]
Using $[\xi_\Gamma]=\alpha-c[\chi]$, $\ddc\Gamma=0$, and
$M=\int_X\kappa\wedge\Gamma$, we obtain
\begin{align*}
 \int_X(\alpha-c\chi-\rho[\kappa])\wedge\Gamma
 &=\int_X\xi_\Gamma\wedge\Gamma-\rho M\\
 &=\int_X(\xi_\Gamma)_{\ac}\wedge\Gamma
   +\int_X(\xi_\Gamma)_{\text{sing}}\wedge\Gamma-\rho M\\
 &\ge\int_X(\xi_\Gamma)_{\ac}\wedge\Gamma   +\int_\Sigma(\xi_\Gamma)_{\text{sing}}\wedge\Gamma-\rho M\\
 &\ge\int_X(\xi_\Gamma)_{\ac}\wedge\Gamma.
\end{align*}
 Applying Proposition~\ref{prop:duality}
for all such $\Gamma$, we obtain a closed positive current
$R\in\alpha-c[\chi]-\rho[\kappa]$ with $R_{\ac}\in K(\chi)$.
Setting $Q=R+c\chi$ proves \eqref{eq:fullconcentratedcurrent}.
Its singular part equals that of $R$ and is positive.
\end{proof}

\subsection{Construct a current in the limiting class}

The preceding construction assumes $d=n$ and a smooth admissible representative.
We now consider the case $2\leq d\leq n$ and allow only the approximating classes to have such representatives.

\begin{proposition}\label{prop:endpoint}
{\color{black}Let $2\le d\le n$, and let $g$ be a monic
right-Noetherian polynomial of degree $d$ with constant coefficients.
Suppose $[\alpha_j]\to[\alpha]$ and any $[\alpha_j]$ contains
a smooth closed form in $\Ccal_g(\chi)$. Assume
\[
 V_g([\alpha]):=\int_X\Phi_n^g([\alpha],[\chi])>0.
\]
For any K\"ahler form $\kappa$, we have
\[
 D_g([\alpha],\kappa):=d\int_X\Phi_{n-1}^g([\alpha],[\chi])\wedge[\kappa]>0.
\]
If $0<\delta<V_g([\alpha])/D_g([\alpha],\kappa)$ and
$\theta\in[\alpha]$ is smooth and closed, there are a
quasi-plurisubharmonic function $u$ and finite constants $L,C>0$
such that $T=\theta+\ddc u$ has positive singular part and
\begin{equation}\label{eq:endpointcurrent}
 T-\delta\kappa\ge-C\chi,\qquad
 (T-\delta\kappa)_{\ac}\in
 \overline{\Ccal_{p_L}(\chi)}\subset
 \overline{\Ccal_g(\chi)}\quad\text{a.e.},
\end{equation}
where $p_L(t)=(t+L)^{n-d}g(t)$.
}
\end{proposition}

\begin{proof}
We first work in a class $[\beta]$ containing a smooth closed
form in $\Ccal_g(\chi)$. No such representative is assumed in
the limiting class $[\alpha]$.
For all sufficiently large $L$, Lemma~\ref{lem:geometry} gives
a representative of $[\beta]$ in $\Ccal_{p_L}(\chi)$.
Set
\[
 V_L=\int_X\Phi_n^{p_L}([\beta],[\chi]),\qquad
 D_L=n\int_X\Phi_{n-1}^{p_L}([\beta],[\chi])\wedge[\kappa],\qquad
 \rho_L([\beta])=V_L/D_L.
\]
These numbers are positive. Proposition~\ref{prop:fullconcentration}
gives a closed current $Q_L$ such that
\begin{equation}\label{eq:raisedcurrent}
 [Q_L]=[\beta]-\rho_L([\beta])[\kappa],\qquad Q_L\ge c_L\chi,
 \qquad (Q_L)_{\ac}\in\overline{\Ccal_{p_L}(\chi)}.
\end{equation}
Here $c_L$ is the next-to-leading coefficient of $p_L$.
The singular part of $Q_L$ is positive.

Write $c_g=c_{d-1}$ and set
\[
 B([\beta])=\int_X(\beta-c_g\chi)\wedge\chi^{n-1},\qquad
 M_0=\int_X\kappa\wedge\chi^{n-1}>0.
\]
By \eqref{eq:shiftGamma},
$(Q_L-c_g\chi)\wedge\chi^{n-1}\ge0$ as a measure.
Thus $B([\beta])\ge\rho_LM_0$.
Since $L^{-(n-d)}p_L\to g$ coefficientwise,
\begin{equation}\label{ineq:L^(-n-d)}
 L^{-(n-d)}V_L\to V_g([\beta]),\qquad
 L^{-(n-d)}D_L\to D_g([\beta],\kappa)>0.
\end{equation}
The denominator is positive on a smooth admissible representative.
Letting $L\to\infty$ gives
\begin{equation}\label{eq:denominatorbound}
 B([\beta])D_g([\beta],\kappa)\ge V_g([\beta])M_0.
\end{equation}

Apply \eqref{eq:denominatorbound} to $[\beta]=[\alpha_j]$ and
pass to the limit. Both factors on the left are nonnegative, and
\[
 B([\alpha])D_g([\alpha],\kappa)
 \ge V_g([\alpha])M_0>0.
\]
Hence $D_g([\alpha],\kappa)>0$. In particular,
\[
\frac{V_g([\alpha_j])}{D_g([\alpha_j],\kappa)}
 \longrightarrow\frac{V_g([\alpha])}{D_g([\alpha],\kappa)}.
\]

Choose $\delta<\rho<V_g([\alpha])/D_g([\alpha],\kappa)$.
Take smooth closed $\theta_j\in[\alpha_j]$ with
$\theta_j\to\theta$ smoothly. Fix $j$ so large that
\[
 \frac{V_g([\alpha_j])}{D_g([\alpha_j],\kappa)}>\rho,
 \qquad \theta-\theta_j\ge-(\rho-\delta)\kappa.
\]
For this fixed class $[\alpha_j]$, by ~\eqref{ineq:L^(-n-d)} we can choose a finite $L$ with $\rho_L([\alpha_j])>\rho$. For such fixed class $[\alpha_j]$ and constant $L>0$, we
take $Q_L$ from \eqref{eq:raisedcurrent}. Define
\[
 T=Q_L+\rho_L\kappa+\theta-\theta_j.
\]
Then $[T]=[\alpha]$ and
\[
 T-\delta\kappa=Q_L+\zeta,\qquad
 \zeta=(\rho_L-\delta)\kappa+\theta-\theta_j
       \ge(\rho_L-\rho)\kappa>0.
\]
Then Lemma~\ref{lem:geometry} gives the cone inclusion  $(T-\delta\kappa)_{\ac}\in
 \overline{\Ccal_{p_L}(\chi)}\subset
 \overline{\Ccal_g(\chi)}\,$ almost everywhere.  The lower bound $T-\delta\kappa \geq -C\chi$ follows from
$Q_L\ge c_L\chi$. Finally, the $\partial\bar\partial$ lemma
gives $T=\theta+\ddc u$, where $u$ is
quasi-plurisubharmonic.
\end{proof}

\section{Regularization and gluing}\label{sec:gluing}

We adapt Richberg's regularized maximum construction
\cite{Richberg,DemaillyBook} to a convex spectral set.
For regularization of positive currents and plurisubharmonic
functions, see \cite{DemaillyRegularization,BlockiKolodziej}.
For the gluing step, we follow the gluing strategy developed in
\cite[Section~5]{Chen} and \cite[Section~6]{CLT}.
\subsection{Preservation of linear inequalities under regularized maximum}

Let $K\subset\Herm(n)$ be nonempty, closed, and convex, with
$K+\{P\ge0\}\subset K$. For $B\geq 0$, set
$h_K(B)=\inf_{A\in K}\tr(BA)$. The separation theorem gives
\begin{equation}\label{eq:halfspaces}
 K=\bigcap_{\substack{B\ge0\\h_K(B)>-\infty}}
       \{A:\tr(BA)\ge h_K(B)\}.
\end{equation}
Indeed, if $A_0\notin K$, separation gives
$B\in\Herm(n)\setminus\{0\}$ and $b\in\mathbb R$ such that
\[
 \tr(BA_0)<b\le\tr(BA)\qquad(A\in K).
\]
Fix $A_1\in K$. Since $A_1+tP\in K$ for any $P\ge0$ and $t\ge0$, we have
\(
 b\le\tr(BA_1)+t\tr(BP),\) hence \(\tr(BP)\ge0.\) for any $P\geq0$.
Thus $B\ge0$ and
$\tr(BA_0)<b\le h_K(B)$. This proves \eqref{eq:halfspaces}.

Fix a constant Hermitian form $\upsilon$ in a coordinate domain.
Let $T$ be a current of order zero with positive singular part,
and write its matrix measure as $A\,dx+dM_{\mathrm s}$, with
$dM_{\mathrm s}\ge0$. Then $A\in K(\upsilon)$ almost everywhere
if and only if
\begin{equation}\label{eq:supportdistribution}
 \tr(BT)\ge h_{K(\upsilon)}(B)\,dx
\end{equation}
in the distribution sense, for every $B\ge0$ in the countable family for $K(\upsilon)$.
Locally if $T=\sqrt{-1}\partial\bar{\partial}p$, consider the conditions
\begin{equation}\label{eq:local-linear-tests}
 \ddc p-\eta\upsilon\ge0,\qquad
 L_B p:=\tr(B\ddc p)\ge h_{K(\upsilon)}(B)+\eta\tr(B\upsilon).
\end{equation}
They are preserved by translation and normalized nonnegative
convolution.

They are also preserved by a convex regularized maximum.
Fix an even nonnegative function
$\vartheta\in C_c^\infty((-1,1))$ with
$\int_{\mathbb R}\vartheta=1$. For $\eta>0$, define
\[
 M_\eta(t_1,\ldots,t_N)
 =\int_{\mathbb R^N}
   \max_{1\le j\le N}(t_j+\eta s_j)
   \prod_{j=1}^N\vartheta(s_j)\,ds_1\cdots ds_N.
\]  The
construction is in \cite[Chapter I, Lemma 5.18]{DemaillyBook},
which is part of Richberg's regularization method \cite{Richberg}. One has $ \frac{\partial M_\eta}{\partial t_j}\ge0$, $\sum_j \frac{\partial M_\eta}{\partial t_j}=1$, and
$\left(\frac{\partial^2 M_\eta}{\partial t_{j}\partial{k}}\right)\ge0$; hence if $p_1,\ldots,p_N$ satisfies~\eqref{eq:local-linear-tests}, then \begin{equation}\label{eq:maxhessian}
 L_BM_\eta(p_1,\ldots,p_N)
 =\sum_j \frac{\partial M_\eta}{\partial t_j}L_Bp_j
       +\tr\left(B\Big(\sqrt{-1}\sum_{j,k} \frac{\partial^2 M_\eta}{\partial t_{j}\partial{k}}
                              \partial p_j\wedge\bar\partial p_k\Big)\right)
\end{equation}
also satisfies \eqref{eq:local-linear-tests}.

\subsection{The gluing construction}

We regularize the current outside a Lelong level set.
Near that set we keep a prescribed smooth potential.
The following theorem glues these potentials so that the resulting
potential agrees with the prescribed smooth potential near the
Lelong level set. We follow the argument of~\cite{Chen}.

\begin{theorem}\label{thm:gluing}
Let $(Y,\chi)$ be a compact K\"ahler manifold of dimension $n$. Let
$K\subset\Herm(n)$ be closed, convex, and unitarily invariant.
Assume $\Int K\ne\varnothing$ and $K+\{P\ge0\}\subset K$.
Let $T=\theta+\ddc u$, where $\theta$ is smooth and closed and
$u$ is quasi-plurisubharmonic. Suppose that, for some $\eta>0$,
\begin{equation}\label{eq:gluinginput}
 T-\eta\chi\ge0,\qquad
 (T-\eta\chi)_{\ac}\in K(\chi)\quad\text{a.e.}
\end{equation}
There is $c_*>0$ with the following property. Let $0<c\le c_*$
and $E_c=\{x:\nu(T,x)\ge c\}$. If  $v\in C^{\infty}(O)$ for a neighborhood  $O$ of
$E_c$ and satisfies $\theta+\ddc v\in\Int K(\chi)$ on $O$, then there is
$w\in C^\infty(Y)$ such that
\[
 \theta+\ddc w\in\Int K(\chi),\qquad w=v\quad\text{near }E_c.
\]

More generally, assume \eqref{eq:gluinginput} and let
$E\subset E_c$ be closed. Let $\varrho$ be a smooth real
$(1,1)$-form satisfying $0\le\varrho\le\chi$ and
$\varrho>0$ on $Y\setminus E$.
If $v\in C^\infty(O)$ for a neighborhood $O$ of $E_c$ and
\(
 \theta+\ddc v\in\Int K(\varrho)
\) on \(O\setminus E\),
then there exists $w\in C^\infty(Y)$ such that
\[
 \theta+\ddc w\in\Int K(\varrho)
 \quad\text{on }Y\setminus E,
 \qquad w=v\quad\text{near }E_c.
\]
\end{theorem}

\begin{proof}
Since $T$ is closed and positive, $E_c$ is a proper analytic
subset for $c>0$ by Siu's theorem \cite{Siu}.

If $A\ge0$ and $0<\upsilon\le\gamma$, the min-max principle gives
\[
 \lambda_j(\upsilon^{-1}A)\ge\lambda_j(\gamma^{-1}A),
 \qquad j=1,\ldots,n,
\] where $\lambda_j$ denotes the $j$-th smallest eigenvalue counted with multiplicity.
Thus $A\in K(\gamma)$ implies $A\in K(\upsilon)$.
If $A>0$ and $A\in\Int K(\gamma)$, the same comparison gives $A\in\Int K(\upsilon)$.
Choose $R>0$ with $RI\in\Int K$. If
$(1-e)\gamma\le\upsilon\le\gamma$ and
$A\ge0$ belongs to $K(\upsilon)$, then
$A/(1-e)\in K(\gamma)$, so
\begin{equation}\label{eq:metricloss}
 A+eR\gamma
 =(1-e)\frac{A}{1-e}+eR\gamma\in\Int K(\gamma).
\end{equation}
Fix $0<e<1/4$ with $eR<\eta/8$.

Take finitely many coordinate domains $U_i$ with constant
Hermitian forms $\upsilon_i$ satisfying
$(1-e)\chi\le\upsilon_i\le\chi$.
Let $\ddc h_i=\theta$ and $p_i=u+h_i$. By
\eqref{eq:gluinginput}, metric comparison, and the property of $K$,
\begin{equation}\label{eq:frozenineq}
 \ddc p_i-\eta\upsilon_i\ge0,\qquad
 (\ddc p_i-\eta\upsilon_i)_{\ac}\in K(\upsilon_i).
\end{equation}
In particular, $p_i$ is plurisubharmonic.

Choose $V_i\Subset D_i\Subset U_i$ so that the $V_i$ cover $Y$.
On each $D_i$, for any small $0<r<\frac{1}{4}\text{dist}(D_i,\partial U_i)$, we set the functions
\[
 M_i(x,r)=\sup_{|a|\le r}p_i(x+a),\qquad
 v_i^r=M_i(\cdot,r)-h_i.
\]
The function $M_i(x,r)$ is continuous in $x$ for fixed $r>0$, and is increasing
and convex in $\log r$; see
\cite[Section 2]{DemaillyRegularization}.

 $M_i(\cdot,r)$ also satisfies \eqref{eq:frozenineq}.
Indeed, consider the normalized nonnegative radial convolution 
$p_{i,\varepsilon}$ and let
$q_\varepsilon=\sup_{|a|\le r}p_{i,\varepsilon}(\cdot+a)$.
The submean inequality gives
\(
 M_i(\cdot,r)\le q_\varepsilon\le M_i(\cdot,r+\varepsilon).
\)
On each compact subset, choose finitely many translations with
\(
 0\le q_\varepsilon-
       \max_{1\le\nu\le N}p_{i,\varepsilon}(\cdot+a_\nu)\le\delta.
\)
By~\eqref{eq:maxhessian}, their regularized maximum
\[
 Q_{\varepsilon,\delta,\tau}
 =M_\tau\bigl(p_{i,\varepsilon}(\cdot+a_1),\ldots,
             p_{i,\varepsilon}(\cdot+a_N)\bigr)
\]
satisfies \eqref{eq:frozenineq} and
\[
 M_i(\cdot,r)-\delta
 \le Q_{\varepsilon,\delta,\tau}
 \le M_i(\cdot,r+\varepsilon)+\tau.
\]
Letting $\varepsilon,\delta,\tau\downarrow0$ proves
\eqref{eq:frozenineq} for $M_i$.

Let $s_{i,\varepsilon}=M_i(\cdot,r)*\rho_\varepsilon$,
where  $\rho_\varepsilon\ge0$ is a smooth convolution kernel
supported in $B(0,\varepsilon)$ with integral one. Since $\upsilon_i$ is constant form, we write
\(
 \sigma:=\ddc s_{i,\varepsilon}
 =A+\eta\upsilon_i,\) where \(
 A:=\bigl(\ddc M_i(\cdot,r)-\eta\upsilon_i\bigr)*\rho_\varepsilon.
\)
The preserved inequalities give $A\ge0$ and $A\in K(\upsilon_i)$.
By \eqref{eq:metricloss},
\begin{equation}\label{eq:marginreturn}
 \begin{aligned}
 \sigma-\frac\eta2\chi
 &=(A+eR\chi)
   +\left(\eta\upsilon_i-\frac\eta2\chi-eR\chi\right)
   \in\Int K(\chi),\\
 \sigma&\ge\eta(1-e)\chi.
 \end{aligned}
\end{equation}

 Fix a common radius $R_0$, set\[
 a_i(x,r)=\frac{M_i(x,R_0)-M_i(x,r)}{\log(R_0/r)}.
\]
With Lelong numbers normalized by
$\nu(T,x)=\lim_{r\downarrow0}M_i(x,r)/\log r$, convexity gives
\(
 a_i(x,r)\searrow\nu(T,x)\) as \(r\downarrow0\).
Coordinate comparison and convexity give constants $L>1$ and $B$
such that, on $D_i\cap D_j$,
\begin{equation}\label{eq:overlap}
 v_i^r\le v_j^{Lr}+Br,\qquad
 v_i^r-v_j^r\le(\log L)a_j(x,Lr)+Br.
\end{equation}
Here the first inequality follows from inclusion of coordinate
balls and the smoothness of $h_i-h_j$; the second follows from
\[
 \frac{M_j(x,Lr)-M_j(x,r)}{\log L}
 \le
 \frac{M_j(x,R_0)-M_j(x,Lr)}{\log(R_0/(Lr))}
 =a_j(x,Lr).
\]

Take smooth cutoff functions $-1\le b_i\le0$, equal to $0$ on $V_i$
and to $-1$ near $\partial D_i$, with
$\ddc b_i\ge-B_0\chi$ for some $B_0\ge1$. Choose
\[
 A_*>2\log L+3,\qquad
 0<c_*\le\frac{\eta}{4A_*B_0},\qquad
 g_i=A_*cb_i.
\]
Then $\ddc g_i\ge-\eta\chi/4$, and these choices are independent
of $v$.

Suppose first that $E_c\ne\varnothing$. Choose
$E_c\subset W\Subset U\Subset O$, where $v$ is defined on $O$,
and numbers
\[
 \max_{Y\setminus W}\nu(T,\cdot)<a'<a<c.
\]
Monotonicity and compactness give
$a_i(x,Lr)\le a'$ outside $W$ for all sufficiently small $r$.

Let
\[
 F_i^r=v_i^r+g_i-a\log r
      =M_i(\cdot,r)-h_i+g_i-a\log r.
\]
By the definition of $a_i$,
\[
 M_i(x,r)=M_i(x,R_0)-a_i(x,r)\log(R_0/r).
\]
For sufficiently small $r>0$, using $a_i(x,r)\le a'$ near
$\partial U$ and $a_i(x,r)\ge c$ on $E_c$, together with
the uniform bounds for $M_i(x,R_0)$, $h_i$, $g_i$, and $v$,
we obtain for $x\in O$,
\[
 \begin{aligned}
 F_i^r-v
 &\ge(a'-a)\log r-C\longrightarrow+\infty
 &&\text{near }\partial U,\\
 v-F_i^r
 &\ge(a-c)\log r-C\longrightarrow+\infty
 &&\text{on }E_c.
 \end{aligned}
\]

Near $\partial D_i$, choose $j$ with $x\in V_j$.
If $a_j(x,Lr)>2c$, then $x\in W$ and
\[
 F_i^r\le v_j^{Lr}-a\log r+Br
       \le(2c-a)\log r+C,\qquad
 v-F_i^r\ge(a-2c)\log r-C\longrightarrow+\infty.
\]
If $a_j(x,Lr)\le2c$, then
\[
 F_j^r-F_i^r
 \ge A_*c-(\log L)a_j(x,Lr)-Br
 \ge(A_*-2\log L)c-Br>2c
\]
once $Br<c$. Choose $r>0$ sufficiently small so that $Br<c$
and each of the preceding lower bounds tending to $+\infty$
is greater than $4c$.
For this fixed $r$, continuity and the finiteness of the cover
give a neighborhood of $E_c$ on which
$v-F_i^r>4c$ for every available index $i$.

Take smooth convolutions $s_i$ of $M_i(\cdot,r)$ satisfying
$\lvert s_i-M_i(\cdot,r)\rvert<c/10$ near $\overline{D_i}$,
and set
\[
 f_i=s_i-h_i+g_i-a\log r.
\]
Then
\(
 |f_i-F_i^r|=|s_i-M_i(\cdot,r)|<c/10.
\)
Since $a\log r$ is constant in the spatial variables,
Equation~\eqref{eq:marginreturn} gives
\[
 \theta+\ddc f_i\in\Int K(\chi),\qquad
 \theta+\ddc f_i
 \ge\eta\left(\frac34-e\right)\chi>\frac\eta2\chi.
\]
For $0<\tau<c/10$ and $I(x)=\{i:x\in D_i\}$, define
\[
 w(x)=
 \begin{cases}
 M_\tau\bigl((f_i(x))_{i\in I(x)},v(x)\bigr),
       &x\in U,\\[2mm]
 M_\tau\bigl((f_i(x))_{i\in I(x)}\bigr),
       &x\in Y\setminus U.
 \end{cases}
\]
To show $w(x)$ is smooth, recall the property of regularized maximum
\cite[Chapter~I, Lemma~5.18(c)]{DemaillyBook}.
That is, if \(
 t_i+2\tau<\max_{j\ne i}t_j\), then \(M_\tau(t_1,\ldots,t_N)=M_\tau\bigl((t_j)_{j\ne i}\bigr).
\)
For $x_0\in\partial D_i$, the preceding estimates give
\[
 f_i+2\tau<f_j
 \quad\text{for some }j\text{ with }x_0\in V_j,
 \qquad\text{or}\qquad
 f_i+2\tau<v
 \quad\text{with }x_0\in W,
\]
on a sufficiently small neighborhood of $x_0$.
Likewise, for $x_0\in\partial U$,
\[
 v+2\tau<f_j
 \quad\text{near }x_0
 \quad\text{for some }j\text{ with }x_0\in V_j.
\]
Thus every branch whose domain has $x_0$ on its boundary
can be deleted, while its dominating branch remains available
near $x_0$. Since there are finitely many domains, $w$ locally
equals the regularized maximum of a fixed finite family of
smooth functions. Hence $w\in C^\infty(Y)$.

Near $E_c$, we have $f_i+2\tau<v$ for every
available index $i$. Therefore
\[
 w=M_\tau(v)
   =\int_{\mathbb R}(v+\tau s)\vartheta(s)\,ds
   =v.
\]
Finally, by \eqref{eq:maxhessian},
$\theta+\ddc w\in\Int K(\chi)$.

If $E_c=\varnothing$, take $a=0$ in the definitions of
$F_i^r$ and $f_i$. Compactness gives $a_i(x,Lr)<c$ everywhere
for small $r$. The same boundary comparison applies, and we set
\[
 w(x)=M_\tau\bigl((f_i(x))_{i\in I(x)}\bigr).
\]
Then $w(x)$ is smooth and satisfies the required conclusion.

For the relative assertion, each $\theta+\ddc f_i$ is positive
definite and belongs to $\Int K(\chi)$ on $D_i$.
Since $0<\varrho\le\chi$ on $Y\setminus E$, metric comparison gives
\[
 \theta+\ddc f_i\in\Int K(\varrho)
 \qquad\text{on }D_i\setminus E.
\]
By assumption,
\(
 \theta+\ddc v\in\Int K(\varrho)
 \) on \(O\setminus E\). By convexity, the upper-set
property, and \eqref{eq:maxhessian}, the same construction gives
\[
 \theta+\ddc w\in\Int K(\varrho)
 \qquad\text{on }Y\setminus E.
\]
The domination near $E_c$ still gives $w=v$ there.
\end{proof}

\section{Numerical perturbation and local extension}\label{sec:local}

We construct the forms contained in $C_g$  near analytic subsets by induction on
their dimension. As in the arguments of Demailly--P\u{a}un,
Collins--Tosatti and Chen \cite{DP, CollinsTosatti, Chen}, we first work on a resolution.  We construct them on resolutions
and then extend them in the normal directions.

The following lemma
provides the required perturbation on the resolution.
On an $m$-dimensional manifold, $\Phi_q^g$ is defined with that
dimension $m$.

\begin{lemma}\label{lem:quantlift}
Let $W$ be a compact K\"ahler manifold of dimension $m$, and let
\(
 g(t)=t^e-\sum_{i=0}^{e-1}\binom ei c_it^i
\)
be right-Noetherian, where $1\le e\le m$. Set $k=m-e$.
Let $\xi,\gamma$ be smooth closed real $(1,1)$-forms satisfying
\(
 \gamma\ge0\) and  \(-C_\xi\gamma\le\xi\le C_\xi\gamma,
\)
and let $\rho$ be K\"ahler. Assume that
\begin{equation}\label{eq:liftproper}
 \int_Z\rho^{p-q}\wedge\Phi_q^g(\xi,\gamma)\ge0,
 \qquad k\le q\le p,
\end{equation}
for every proper irreducible analytic subvariety $Z\subset W$
of dimension $p$, and that
\begin{equation}\label{eq:liftwhole}
 J_j:=\frac{e!}{(e-j)!}
       \int_W\rho^j\wedge\Phi_{m-j}^g(\xi,\gamma)>0,
 \qquad 0\le j\le e.
\end{equation}
Set $r=\max\{1,e-1\}$ and choose $N>m+kr$.
For sufficiently large fixed $K_0$ and sufficiently small $s>0$, let
\[
 L_s=K_0s^{-r},\qquad
 \xi_s=\xi+s\rho,\qquad
 \gamma_s=\gamma+s^N\rho,\qquad
 p_s(t)=(t+L_s)^kg(t).
\]
Then $p_s$ is right-Noetherian, and
\begin{equation}\label{eq:liftmargin}
 \int_Z\rho^{p-q}\wedge\Phi_q^{p_s}(\xi_s,\gamma_s)
 \ge\frac12s^q\int_Z\rho^p,
 \qquad 0\le q\le\min\{p,m-1\},
\end{equation}
for every irreducible analytic subvariety $Z\subset W$,
including $W$. The choices of $K_0$ and $s$ are independent of $Z$.
Moreover, with
\[
 V_s=\int_W\Phi_m^{p_s}(\xi_s,\gamma_s),\qquad
 D_s=m\int_W\rho\wedge\Phi_{m-1}^{p_s}(\xi_s,\gamma_s),
\]
we have
\begin{equation}\label{eq:liftlimits}
 L_s^{-k}V_s\longrightarrow J_0,\qquad
 L_s^{-k}D_s\longrightarrow J_1.
\end{equation}
The translation eliminating the coefficient of degree $m-1$ is
\(
 c_s=\frac{ec_{e-1}-kL_s}{m}
\). We have \(|c_s|s^N\to 0.
\)
\end{lemma}

\begin{proof}
Right-Noetherianity of $p_s(t)$ follows from
\cite[Lemma~3.1]{CNX}. 
Write $g(t)=\sum_{i=0}^e a_it^i$, with $a_e=1$, and set
\[
 \Xi_a^L=\Phi_a^{(t+L)^kg}(\xi,\gamma),
 \qquad 0\le a\le m.
\]
Expanding $(t+L)^kg(t)$ and differentiating $m-a$ times gives
\begin{equation}\label{eq:liftcoefficients}
 \Xi_a^L
 =\sum_{\substack{0\le h\le k,\ 0\le i\le e\\i+h\ge m-a}}
   \binom kh L^{k-h}a_i
   \frac{a!\,(i+h)!}{m!\,(i+h-m+a)!}
   \xi^{i+h-m+a}\wedge\gamma^{m-i-h}.
\end{equation}
The leading terms correspond to $h=0$ when $a\ge k$, and
to $(h,i)=(k-a,e)$ when $a<k$. Thus
\begin{align}
 \Xi_a^L
 &=\mu_aL^k\Phi_a^g(\xi,\gamma)
       +O(L^{k-1})\gamma^a,
 &&a\ge k,\label{eq:liftlargea}\\
 \Xi_a^L
 &=b_aL^a\gamma^a+O(L^{a-1})\gamma^a,
 &&a<k,\label{eq:liftsmalla}
\end{align}
where
\(
 \mu_a=\frac{a!\,e!}{m!\,(a-k)!}\) and
 \(
 b_a=\frac{k!\,(m-a)!}{m!\,(k-a)!}.
\)
Here $\Xi_0^L=1$. If $k=0$, then $e=m$ and $\mu_a=1$,
so $\Xi_a^L=\Phi_a^g(\xi,\gamma)$.
The bounds  follow from $-C_\xi\gamma\le\xi\le C_\xi\gamma$;
when $\gamma\geq 0$ is degenerate, set $\gamma+\varepsilon\rho$
and let $\varepsilon\downarrow0$.

Fix $Z$ of dimension $p$ and $0\le q\le\min\{p,m-1\}$.
Taylor expansion gives
\begin{equation}\label{eq:lifttaylor}
 \Phi_q^{(t+L)^kg}(\xi+s\rho,\gamma)
 =\sum_{l=0}^q\binom ql s^l\rho^l\wedge\Xi_{q-l}^L.
\end{equation}
If $k=0$, all integrated terms are nonnegative by
\eqref{eq:liftproper} and \eqref{eq:liftwhole}.
If $q<k$, all terms are nonnegative for large $L$ by
\eqref{eq:liftsmalla}.
In both cases, the term $l=q$ contributes $s^q\int_Z\rho^p$.

It remains to consider $k>0$ and $q\ge k$.
Let $a=q-k$ and
\(
 I_Z=\int_Z\gamma^k\wedge\rho^{p-k}\ge0.
\)
For $l\le a$, the integrated leading terms in
\eqref{eq:liftlargea} are nonnegative by the hypotheses.
The contribution with $l=a$ is
\[
 \binom qa\mu_kL^ks^aI_Z
 =\mu_qL^ks^aI_Z,
 \qquad \binom q{q-k}\mu_k=\mu_q.
\]
Choose $C\ge1$ with $\gamma\le C\rho$. For $l\le a$,
\[
 \int_Z\rho^{p-q+l}\wedge\gamma^{q-l}
 \le C^{q-l-k}I_Z.
\]
Hence, for $0<s\le1$, the sum of the absolute values of the
integrated errors is bounded by
\[
 C_0L^{k-1}\sum_{l=0}^a\binom ql s^l
       \int_Z\rho^{p-q+l}\wedge\gamma^{q-l}
 \le C_1L^{k-1}I_Z,
\]
where $C_1$ is independent of $Z,p,q,s,L$.
Choose
\(
 K_0>\max\left\{1,\max_{k\le q\le m-1}\frac{2C_1}{\mu_q}\right\}.
\)
Since $a\le e-1\le r$,
$ L_ss^a=K_0s^{a-r}\ge K_0,$ and
 $\mu_qL_s^ks^aI_Z-C_1L_s^{k-1}I_Z
 \ge\frac{\mu_q}{2}L_s^ks^aI_Z\ge0.$
The terms with $l>a$ are nonnegative for small $s$.
In particular, the term $l=q$ remains
$s^q\int_Z\rho^p$, even when $I_Z=0$.
Thus, in all cases,
\[
 \int_Z\rho^{p-q}\wedge\Phi_q^{p_s}(\xi_s,\gamma)
 \ge s^q\int_Z\rho^p.
\]

The forms $\xi_s,\gamma$ are uniformly bounded relative to $\rho$,
and the coefficients of $p_s$ are bounded by $CL_s^k$.
Expanding $\gamma_s=\gamma+s^N\rho$ therefore gives
\[
 -CL_s^ks^N\rho^q
 \le
 \Phi_q^{p_s}(\xi_s,\gamma_s)-\Phi_q^{p_s}(\xi_s,\gamma)
 \le CL_s^ks^N\rho^q.
\]
Since
\[
 \frac{L_s^ks^N}{s^q}
 =K_0^ks^{N-kr-q}\longrightarrow0
 \qquad(0\le q\le m-1),
\]
the perturbation is bounded in both directions by
$\frac12s^q\rho^q$ for sufficiently small $s$, independently of $Z$.
Wedging with $\rho^{p-q}$ and integrating proves
\eqref{eq:liftmargin}.

Finally,
\[
 L_s^{-k}p_s(t)=\left(1+\frac{t}{L_s}\right)^kg(t)
 \longrightarrow g(t)
\]
coefficientwise, while $\xi_s\to\xi$ and $\gamma_s\to\gamma$
smoothly. Consequently,
\(
 L_s^{-k}\Phi_m^{p_s}(\xi_s,\gamma_s)
 \longrightarrow\Phi_m^g(\xi,\gamma),\) and
 \(L_s^{-k}m\Phi_{m-1}^{p_s}(\xi_s,\gamma_s)
 \longrightarrow e\Phi_{m-1}^g(\xi,\gamma).
\)
Integrating gives \eqref{eq:liftlimits}.
The coefficient of $t^{m-1}$ in $p_s$ is $kL_s-ec_{e-1}$, so
\[
 c_s=\frac{ec_{e-1}-kL_s}{m},
 \qquad
 |c_s|s^N=O(s^{N-r})\longrightarrow0,
\]
where $N>r$ follows from $N>m+kr$.
\end{proof}

\subsection{Extending in the normal directions}
Once the tangential condition is satisfied, a large positive Hessian
in the normal directions gives an admissible form near the subvariety.
We need this extension to agree with a prescribed potential near the
singular set.
\begin{lemma}[Local extension]\label{lem:normal}
 Let $\theta$ be a smooth closed real $(1,1)$-form on compact K\"ahler manifold $X^n$. Let $f$ be a monic right-Noetherian polynomial of degree
$1\le d\le n$ with constant real coefficients.
Let $S\subset Z\subset X$ be closed analytic subsets, with $S$
containing the singular loci and the intersections of distinct
irreducible components of $Z$.
Suppose $v$ is smooth near $S$ and
$\theta+\ddc v\in\Ccal_f(\chi)$ there.

For each component $Z_j$, let $p_j=\dim Z_j$ and $c_j=n-p_j$,
and choose a resolution $\pi_j:Y_j\to Z_j$ that is an
isomorphism on $Z_j\setminus\pi_j^{-1}(S)$.
Suppose $u_j\in C^\infty(Y_j,\mathbb R)$ satisfies
$u_j=\pi_j^*v$ near $\pi_j^{-1}(S)$.
Whenever $c_j<d$, assume additionally that
\[
 g_j=\frac{(d-c_j)!}{d!}f^{(c_j)},\qquad
 \pi_j^*\theta+\ddc u_j\in\Ccal_{g_j}(\pi_j^*\chi)
 \quad\text{on }Y_j\setminus\pi_j^{-1}(S).
\]
Then there is a smooth real function $w$ near $Z$ such that
\[
 \theta+\ddc w\in\Ccal_f(\chi),
 \qquad w=v\quad\text{near }S.
\]
\end{lemma}
\begin{proof}
We first establish the matrix calculation used below.
Write $p+c=n$.
When $c<d$, set
\(
 g=\frac{(d-c)!}{d!}f^{(c)}
\)
and assume $A\in\Ccal_g$.
When $c\ge d$, let $A\in\Herm(p)$ be arbitrary.
Let $\lambda$ be the eigenvalue vector of $A$, and set
$F=\Pol_n(f)$.

For tangent indices $I\subset\{1,\ldots,p\}$ and normal
indices $J\subset\{p+1,\ldots,n\}$, let
$a=|I|$ and $b=|J|$.
If $a+b\le d$, the leading term of
$\partial_I\partial_JF(\lambda,L\mathbf1_c)$ in $L$ is
\[
 \begin{cases}
 \displaystyle
 \frac{p!d!}{n!(d-c)!}\,
 \partial_I\Pol_p(g)(\lambda)L^{c-b},
 & c<d,\quad a\le d-c,\\[6pt]
 \displaystyle
 \frac{\binom{c-b}{d-a-b}}{\binom nd}\,
 L^{d-a-b},
 & \text{otherwise}.
 \end{cases}
\]
Both coefficients are positive under the stated assumptions.
Including $I=J=\varnothing$, we obtain positivity of $F$
and all its nonzero partials for sufficiently large $L$.
Then by \eqref{eq:diagonalcomponent},
\[
 \operatorname{diag}(A,LI_c)\in\Ccal_f
 \qquad (L\ge L_0).
\]
On a compact subset of $\Ccal_g$, the leading coefficients
have a positive lower bound and the remaining coefficients
are bounded, so $L_0$ can be chosen uniformly.
When $c\ge d$, the leading coefficients are positive constants
independent of $A$, so the same conclusion holds.

We next descend and extend the potentials.
Choose an open neighborhood $N_j$ of $\pi_j^{-1}(S)$ on which
$u_j=\pi_j^*v$, and set
\[
 O_j=Z_j\setminus\pi_j(Y_j\setminus N_j).
\]
Since $\pi_j$ is proper, $O_j$ is open in $Z_j$.
The definition of $O_j$ and the inclusion
$\pi_j^{-1}(S)\subset N_j$ give
\(
 S\cap Z_j\subset O_j\) and 
 \(
 \pi_j^{-1}(O_j)\subset N_j.
\)
Thus the function
\[
 \bar u_j=
 \begin{cases}
 u_j\circ\pi_j^{-1},&\text{on }Z_j\setminus S,\\
 v,&\text{on }O_j
 \end{cases}
\]
is well-defined and agrees with $v$ near $S\cap Z_j$.
Choose a smooth function $v_0$ near $Z$ equal to $v$ near $S$.
Each $\bar u_j-v_0|_{Z_j}$ has compact support in
$Z_j\setminus S$.
Extend these differences in normal coordinates, using a
partition of unity with supports away from the other components.
This gives a smooth function $\widetilde w$ near $Z$ satisfying
\[
 \widetilde w|_{Z_j}=\bar u_j,\qquad
 \widetilde w=v\quad\text{near }S.
\]
Set $\xi=\theta+\ddc\widetilde w$.
For $c_j<d$, the hypothesis gives
\[
 \pi_j^*\xi
 =\pi_j^*\theta+\ddc u_j
 \in\Ccal_{g_j}(\pi_j^*\chi)
 \quad\text{on }Y_j\setminus\pi_j^{-1}(S).
\]
Since $\pi_j$ is an isomorphism there, the tangential
restriction of $\xi$ to $Z_j\setminus S$ belongs to
$\Ccal_{g_j}(\chi|_{Z_j\setminus S})$.

Choose an open neighborhood $U$ of $S$ on which
$\widetilde w=v$ and $\xi\in\Ccal_f(\chi)$, and let
$Z_0=Z\setminus U$.
If $Z_0=\varnothing$, take $w=\widetilde w$.
Otherwise, choose finitely many coordinate domains
\[
 V_\ell\Subset D_\ell\Subset X\setminus S,
 \qquad Z_0\subset\bigcup_\ell V_\ell,
\]
such that each $D_\ell$ meets only one component of $Z$ and
has holomorphic coordinates $(z_\ell,\zeta_\ell)$ with
\(
 Z\cap D_\ell=\{\zeta_\ell=0\}.
\)
Choose smooth cutoff functions
\[
 \psi_\ell\in C_c^\infty(D_\ell,\mathbb R),\qquad
 0\le\psi_\ell\le1,\qquad
 \psi_\ell=1\quad\text{near }\overline{V_\ell},
\]
and define
\(
 q=\sum_\ell\psi_\ell|\zeta_\ell|^2,
\)
extending each summand by zero outside $D_\ell$.
Then
\[
 q=dq=0\quad\text{on }Z,\qquad
 q=0\quad\text{near }S.
\]

For $x\in Z\cap D_\ell$,
we have $\zeta_\ell(x)=0$ and $d|\zeta_\ell|^2_\ell(x)=0$.
Thus $\ddc q$ is semipositive along $Z$, with zero tangent
and mixed blocks.
Since the $V_\ell$ cover $Z_0$, its normal block is positive
definite there. Compactness gives a uniform lower bound
$\nu I_c$, with $\nu>0$.

On $Z_0\cap Z_j$, write $p=p_j$ and $c=c_j$.
In a $\chi$-unitary frame adapted to the tangent and normal
spaces, the matrices of $\xi$ and $\ddc q$ are
\[
 \operatorname{Mat}(\xi)=
 \begin{pmatrix}A&B\\B^*&D\end{pmatrix},
 \qquad
 \operatorname{Mat}(\ddc q)=
 \begin{pmatrix}0&0\\0&H\end{pmatrix},
 \qquad H\ge\nu I_c.
\]
The tangential condition gives $A\in\Ccal_{g_j}$ when $c_j<d$.
By compactness, we may choose one $\varepsilon>0$ such that
\[
 A-\varepsilon I_{p_j}\in\Ccal_{g_j}
 \quad\text{on }Z_0\cap Z_j,\qquad c_j<d.
\]
For $c_j\ge d$, no tangential condition is required.
The matrix calculation above, applied to the finitely many
compact families of shifted tangent blocks, gives one $L_0>0$
such that
\begin{equation}\label{eq:diag(A,I)}
 \operatorname{diag}(A-\varepsilon I_p,L_0I_c)
 \in\Ccal_f
 \quad\text{on }Z_0.
\end{equation}

For tangent and normal vectors $s,t$, Cauchy- Schwarz inequality gives
\(
 2|\langle Bt,s\rangle|
 \le\varepsilon|s|^2
       +\varepsilon^{-1}\|B\|^2|t|^2.
\)
Consequently,
\[
 \begin{pmatrix}A&B\\B^*&D+MH\end{pmatrix}
 \ge
 \operatorname{diag}\!\left(
 A-\varepsilon I_p,\,
 (M\nu-\|D\|-\varepsilon^{-1}\|B\|^2)I_c
 \right).
\]
Set
\[
 C_0=\sup_{Z_0}
       \bigl(\|D\|+\varepsilon^{-1}\|B\|^2\bigr),
 \qquad
 M>\frac{L_0+C_0}{\nu}.
\]
Then by~\eqref{eq:diag(A,I)}, 
$\xi+M\ddc q\in\Ccal_f(\chi)$ on $Z_0$.
On $Z\cap U$, the same follows from
$\xi\in\Ccal_f(\chi)$ and $\ddc q\ge0$ along $Z$.
By compactness and openness, this condition holds on a
neighborhood of $Z$.
Therefore
\[
 w=\widetilde w+Mq
\]
satisfies $\theta+\ddc w\in\Ccal_f(\chi)$ there and
$w=v$ near $S$.
\end{proof}

\subsection{Induction on analytic subsets}
We now combine the perturbation lemma with relative gluing and
normal extension. The induction on dimension of the whole manifold gives
a form on each resolution. The following induction considers the proper singular subvariety.
\begin{proposition}[Local extension]\label{lem:local}
Let $(X,\chi)$ be a compact K\"ahler manifold of dimension $n\ge2$.
Assume that Theorem~\ref{thm:positive} holds for full-degree
polynomials in dimensions smaller than $n$.
Let $f$ be a monic right-Noetherian polynomial of degree
$1\le d\le n$ with constant real coefficients, and let
$\theta\in[\alpha]$ be a smooth closed real $(1,1)$-form.
Suppose a K\"ahler form $\kappa$ satisfies
\begin{equation}\label{eq:localmixed}
 \int_V[\kappa]^{p-q}\wedge\Phi_q^f([\alpha],[\chi])>0,
 \qquad n-d\le q\le p<n,
\end{equation}
for every proper irreducible analytic subvariety $V\subset X$
of dimension $p$.
Then every proper analytic subset $Z\subset X$ has a neighborhood
carrying a smooth real function $v$ such that
$\theta+\ddc v\in\Ccal_f(\chi)$.
\end{proposition}

\begin{proof}
We induct on the largest dimension $m$ of a component of $Z$.
For $m=0$, $v$ can be chosen as sufficiently large positive quadratic functions
in disjoint coordinate balls.
Assume $m>0$ and set $k=n-d$.
Let $S$ be the union of the lower-dimensional components of $Z$,
the singular loci of the $m$-dimensional components,
and their pairwise intersections. Then $\dim S<m$.

For each $m$-dimensional component $Z_j$, embedded resolution
and principalization \cite{BM} give a compact K\"ahler
resolution $\pi:W\to Z_j$ and an effective divisor $D$ such that
\[
 E:=\operatorname{Supp}D=\pi^{-1}(S\cap Z_j),
 \qquad
 \pi:W\setminus E\xrightarrow{\sim}Z_j\setminus S.
\]
Here $W$ is K\"ahler because it is a smooth submanifold of
an iterated blowup of $X$ along smooth centers.

If $m\le k$, apply the induction hypothesis to $S$.
Extend the pullback of the resulting local potential smoothly
to each resolution, keeping it unchanged near $E$.
Since $n-m\ge d$, Lemma~\ref{lem:normal} requires no tangential
condition and gives the conclusion.

Suppose $m>k$. Set
\[
 e=m-k,\qquad
 g=\frac{e!}{d!}f^{(n-m)},\qquad
 \xi=\pi^*\theta,\qquad \gamma=\pi^*\chi.
\]
Then $g$ is monic and right-Noetherian, and
\[
 \gamma\ge0,\qquad \gamma>0\ \text{on }W\setminus E,
 \qquad -C_\xi\gamma\le\xi\le C_\xi\gamma.
\]
After scaling $\kappa$, assume $\kappa\ge\chi$.
Choose a K\"ahler form $\sigma$ on $W$ and let $\rho=\pi^*\kappa+\zeta\sigma$, where $\zeta>0$ will be fixed
sufficiently small. In particular, $\rho\ge\gamma$.

The restriction identity gives
\begin{equation}\label{eq:localderivativeidentity}
 \Phi_q^g(\xi,\gamma)=\pi^*\Phi_q^f(\theta,\chi),
 \qquad k\le q\le m.
\end{equation}
At $\zeta=0$, the projection formula and \eqref{eq:localmixed}
give
\[
 \int_W(\pi^*\kappa)^{m-q}\wedge\Phi_q^g(\xi,\gamma)
 =\int_{Z_j}\kappa^{m-q}\wedge\Phi_q^f(\theta,\chi)>0.
\]
There are only finitely many indices $k\le q\le m$.
Thus one sufficiently small $\zeta>0$ makes all these
integrals positive with $\pi^*\kappa$ replaced by $\rho$.
This verifies \eqref{eq:liftwhole} in Lemma~\ref{lem:quantlift}.

Let $V\subset W$ be a proper irreducible analytic subvariety
of dimension $p<m$. By Remmert's theorem, $\pi(V)$ is analytic
and has dimension  $\leq p$.
The induction hypothesis gives
$\theta_v=\theta+\ddc v\in\Ccal_f(\chi)$ near $\pi(V)$.
Then Stokes' theorem and \eqref{eq:localderivativeidentity} give
\[
 \int_V\rho^{p-q}\wedge\Phi_q^g(\xi,\gamma)
 =\int_V\rho^{p-q}\wedge\pi^*\Phi_q^f(\theta_v,\chi)
 \ge0,\qquad k\le q\le p.
\]
The last inequality follows from strong positivity and remains
valid when $V\subset E$. Thus \eqref{eq:liftproper} in Lemma~\ref{lem:quantlift} also holds.

Suppose first that $m\ge2$.
Lemma~\ref{lem:quantlift} gives full-degree polynomials $p_s$,
forms $\xi_s=\xi+s\rho$ and $\gamma_s=\gamma+s^N\rho$,
strict mixed inequalities, and positive top integrals.
Since $m<n$, by the assumption of Theorem~\ref{thm:positive} in dimension $m$, we can obtain a smooth form in $[\xi_s]\cap\Ccal_{p_s}(\gamma_s)$.

Let $V_s>0$ and $D_s>0$ be defined as in Lemma~\ref{lem:quantlift}. Set $\delta=J_0/(4J_1)$ and fix $s>0$ sufficiently small that
\[
 r_s:=\frac{V_s}{D_s}>2\delta,\qquad
 s+|c_s|s^N<\frac{\delta}{4},\qquad s^N\le1.
\]
Write $c=c_s$, $t=s^N$, and $h=\gamma+t\rho$.
Then $\gamma\le h\le2\rho$.
All subsequent constructions use this fixed $s$.

Proposition~\ref{prop:fullconcentration} gives a closed current
$Q_s$ satisfying
\[
 [Q_s]=[\xi_s]-r_s[\rho],\qquad
 Q_s\ge ch,\qquad
 (Q_s)_{\ac}\in\overline{\Ccal_{p_s}(h)}.
\]
Define currents
\[
 T=Q_s+r_s\rho,\qquad
 R=T-\delta\rho-ch
   =(Q_s-ch)+(r_s-\delta)\rho\ge0.
\]
Then Lemma~\ref{lem:geometry} gives
\[
 R_{\ac}\in
 \overline{\Ccal_{p_s(y+c)}(h)}
 \subset K(h),
 \qquad K:=\overline{\Ccal_{g(y+c)}}.
\]
Now set
\begin{equation}\label{eq:localidentity}
 Q:=T-s\rho-c\gamma
   =R+(\delta-s+ct)\rho,
 \qquad [Q]=[\xi-c\gamma].
\end{equation}
Since $\delta-s+ct\ge3\delta/4$, taking $\eta_0=\delta/4$ gives
\[
 \begin{aligned}
 Q-\eta_0h
 &=R+(\delta-s+ct)\rho-\eta_0h\ge R+\frac{\delta}{4}\rho\ge0,
 \end{aligned}
 \qquad
 (Q-\eta_0h)_{\ac}\in K(h).
\]

If $m=1$, then $k=0$. Write $g(y)=y-c$ and set $K=[0,\infty)$.
Since $\int_W(\xi-c\gamma)>0$, the Poisson equation gives a
smooth positive form $Q\in[\xi-c\gamma]$.
Take $h=\rho\ge\gamma$ and choose $\eta_0>0$ with
$Q-\eta_0h>0$.

We next arrange positive Lelong numbers along $E$.
If $E\ne\varnothing$, choose a logarithmic potential $\ell_D$
such that
\[
 \ddc\ell_D=[D]-\Theta_D,
\]
where $\Theta_D$ is smooth.
Choose $\tau>0$ sufficiently small that
$\eta_0h/2-\tau\Theta_D\ge0$, and put
$Q'=Q+\tau\ddc\ell_D$. Then
\[
 Q'-\frac{\eta_0}{2}h
 =(Q-\eta_0h)
   +\left(\frac{\eta_0}{2}h-\tau\Theta_D\right)
   +\tau[D]\ge0,
\]
and
\[
 \left(Q'-\frac{\eta_0}{2}h\right)_{\ac}\in K(h),
 \qquad
 \nu(Q',x)\ge\tau\nu([D],x)\quad(x\in E).
\]
The last quantity has a positive lower bound on $E$.
If $E=\varnothing$, simply take $Q'=Q$.
In both cases, $[Q']=[\xi-c\gamma]$, so the $\ddc$-lemma gives
\[
 Q'=\xi-c\gamma+\ddc\varphi
\]
for a quasi-plurisubharmonic function $\varphi$ on $W$.

Perform this construction on every resolution.
For each resulting current $Q'_j$, choose a positive Lelong
threshold $a_j$ allowed by Theorem~\ref{thm:gluing}, small enough
that
\[
 E_j\subset
 F_j:=\{x\in W_j:\nu(Q'_j,x)\ge a_j\}.
\]
By Siu's theorem, each $F_j$ is a proper analytic subset.
Remmert's theorem shows that
\[
 S':=S\cup\bigcup_j\pi_j(F_j)
\]
is analytic and has dimension less than $m$.
The induction hypothesis supplies one smooth potential $v$
near $S'$ with $\theta+\ddc v\in\Ccal_f(\chi)$.

On each resolution, restriction to the tangent spaces gives
\[
 \xi+\ddc\pi^*v\in\Ccal_g(\gamma),
 \qquad
 \xi-c\gamma+\ddc\pi^*v\in\Int K(\gamma)
 \quad\text{near }F_j\text{ outside }E_j.
\]
Since $0\le\gamma\le h$ and $\gamma>0$ outside $E_j$,
the relative assertion of Theorem~\ref{thm:gluing} gives
$u_j\in C^\infty(W_j,\mathbb R)$ such that
\[
 u_j=\pi_j^*v\quad\text{near }F_j,\qquad
 \pi_j^*\theta+\ddc u_j
 \in\Ccal_g(\pi_j^*\chi)
 \quad\text{on }W_j\setminus E_j.
\]

Finally, $\pi_j$ is an isomorphism off $E_j\subset F_j$,
and distinct components meet only in $S$. Hence
\[
 \pi_j^{-1}(S'\cap Z_j)
 =\pi_j^{-1}\bigl((S\cap Z_j)\cup\pi_j(F_j)\bigr)
 \subset F_j.
\]
Thus $u_j=\pi_j^*v$ near the full inverse image of $S'\cap Z_j$.
On the lower-dimensional components, which lie in $S'$,
use the pullback of $v$.
Lemma~\ref{lem:normal}, applied with $S'$ in place of $S$,
now gives the required potential near $Z$.
\end{proof}

We clarify the order of the two inductions.
In the following, we fix $n$ and assume Theorem~\ref{thm:positive} for full-degree
polynomials in every dimension smaller than $n$.
Under this assumption, Proposition~\ref{lem:local} is proved
by induction on the largest dimension $m$ of an analytic
subset of the fixed manifold $X$.

Once Proposition~\ref{lem:local} has been established for
all $m<n$, it supplies the local hypothesis of
Proposition~\ref{lem:global}.
In the proof of Theorem~\ref{thm:positive}, the mixed
inequalities give positivity of the top integral along
$[\alpha]+t[\kappa]$, so Proposition~\ref{lem:global}
yields the global representative in dimension $n$.

\section{Global existence and strict numerical positivity}\label{sec:global}
{
We use a continuity argument along a ray of classes.
At its limiting point, mass concentration gives a current.
Local extension and gluing then give a smooth admissible representative.

\begin{proposition}\label{lem:global}
Let $g$ be a constant right-Noetherian polynomial of
degree $2\le d\le n$, and let $\theta\in[\alpha]$ be smooth and
closed. Suppose that, near any proper analytic set,
there is a smooth function $v$ with $\theta+\ddc v\in\Ccal_g(\chi)$.
If a K\"ahler form $\lambda$ satisfies
\[
 \int_X\Phi_n^g([\alpha]+t[\lambda],[\chi])>0\qquad(t\ge0),
\]
then $[\alpha]$ contains a smooth closed form in $\Ccal_g(\chi)$.
\end{proposition}

{\color{black}
\begin{proof}
Let
\[
 I=\{t\ge0:([\alpha]+t[\lambda])\cap\Ccal_g(\chi)
                       \ne\varnothing\}.
\]
The cone is open and is preserved by addition of positive forms.
{For large $t$, the form $\theta+t\lambda$ belongs to
$\Ccal_g(\chi)$ by \eqref{eq:diagonalcomponent}.}
Thus $I$ is a nonempty open upper interval in $[0,\infty)$.
Set $t_* = \inf I$ and choose $t_j\in I$ with $t_j\downarrow t_*$.
The top integral at $[\eta]=[\alpha]+t_*[\lambda]$ is positive.
Proposition~\ref{prop:endpoint} gives a current $T\in[\eta]$
and constants $\delta>0$ and $C\ge0$ such that
\[
 T-\delta\chi\ge-C\chi,
 \qquad (T-\delta\chi)_{\ac}\in\overline{\Ccal_g(\chi)}.
\]
By the $\partial\bar\partial$ lemma, we can write
$T=\theta_*+\ddc u$, where $u$ is quasi-plurisubharmonic.
Set
\[
 S=T+C\chi,
 \qquad K(\chi)=\overline{\Ccal_g}(\chi)+C\chi.
\]
By Lemma~\ref{lem:geometry}, $K(\chi)$ is closed, convex and
unitarily invariant, has nonempty interior, and is stable
under addition of semipositive matrices. We also have
\[
 S-\delta\chi\ge0,
 \qquad
 (S-\delta\chi)_{\ac}\in K(\chi)
 \quad\text{a.e.}
\]

Choose $0<c\le c_*$ as in Theorem~\ref{thm:gluing}, and put
\[
 E_c=\{x\in X:\nu(S,x)\ge c\}.
\]
This is a proper analytic subset of $X$.
By the  hypothesis near any proper analytic set, there is a smooth function $v$
near $E_c$ such that
$\theta+\ddc v\in\Ccal_g(\chi)$.
Since $t_*\lambda\ge0$, we obtain
\[
 \theta_*+C\chi+\ddc v
 =\theta+\ddc v+t_*\lambda+C\chi
 \in\operatorname{Int}K(\chi)
\]
near $E_c$.

Theorem~\ref{thm:gluing}, applied to
$S=\theta_*+C\chi+\ddc u$, now gives a global smooth function
$w$ satisfying
\[
 \theta_*+C\chi+\ddc w\in\operatorname{Int}K(\chi).
\]
Since $\operatorname{Int}K=\Ccal_g+CI$, it follows that
\[
 \theta_*+\ddc w\in\Ccal_g(\chi).
\]
Thus $t_*\in I$. If $t_*>0$, openness of $I$ gives
$t_*-\varepsilon\in I$ for some $0<\varepsilon<t_*$,
contrary to the definition of $t_*$.
Therefore $t_*=0$, and $\theta+\ddc w$ is the required
smooth closed form in $\alpha$.
\end{proof}

Now we prove the following numerical criterion on compact K\"ahler manifolds.
\begin{theorem}\label{thm:positive}
Let $g$ be a constant monic right-Noetherian polynomial of
degree $1\le d\le n$. Suppose
\begin{equation}\label{eq:positive-top}
 \int_X\Phi_n^g([\alpha],[\chi])>0
\end{equation}
and suppose one K\"ahler form $\kappa$ satisfies
\eqref{eq:mixed} with $g$ in place of $f$. Then $\alpha$ contains
a smooth closed form in $\Ccal_g(\chi)$.
\end{theorem}

\begin{proof}
We argue by induction on the dimension of $X$.
 The case
$n=1$ is clear. We fix $n\ge2$ and $d\ge2$ and assume the theorem in
smaller dimensions. In particular, we may apply
Proposition~\ref{lem:local}. Then the strict mixed inequalities on
proper subvarieties give a local potential $v$ with
$\theta+\ddc v\in\Ccal_g(\chi)$ near any proper compact
analytic subset of $X$.

For the classes $[\alpha+t\kappa]$, $t\ge0$, Taylor expansion gives
\begin{equation}\label{eq:top-ray}
 \begin{split}
 \int_X\Phi_n^g([\alpha+t\kappa],[\chi])
 &=\int_X\Phi_n^g([\alpha],[\chi])+\sum_{j=1}^d\binom dj t^j
       \int_X[\kappa]^j\Phi_{n-j}^g([\alpha],[\chi]).
 \end{split}
\end{equation}
For the $j$th term, set $p=n$ and $q=n-j$. Since
$n-d\le q\le n-1$, its coefficient is one of the mixed numbers
assumed positive on $X$. The constant term is positive by
\eqref{eq:positive-top}. Thus the entire expression is positive
for any $t\ge0$. Proposition~\ref{lem:global} then gives a global form in
$\alpha\cap\Ccal_g(\chi)$, completing the induction.
\end{proof}

The mixed inequalities also hold for $\alpha+t[\kappa]$ whenever
$t\ge0$. By \eqref{eq:derivativetaylor},
\begin{equation}\label{eq:mixed-ray}
 \Phi_q^g([\alpha+t\kappa],[\chi])
 =\sum_{j=0}^{d-n+q}\binom{d-n+q}{j}t^j[\kappa]^j
 \Phi_{q-j}^g([\alpha],[\chi]).
\end{equation}
After multiplication by $[\kappa]^{p-q}$, the $j$th coefficient
is the original test with index $q-j$. This index remains between
$n-d$ and $\min\{p,n-1\}$. Thus any coefficient used in the
expansion is positive by \eqref{eq:mixed}. 

\subsection{From strict inequalities to uniform stability}

We will use the inclusion $\Ccal_g\subset\mathcal D_g$ and
the positivity of $\Phi_q^g$ on $\Ccal_g$ from
Lemma~\ref{lem:geometry}.

\begin{proof}[Proof of Theorem~\ref{thm:main}]
Assume first that\/ {\rm (i)} holds. 
Suppose $d\ge2$, and let $r_1=r(f')$.
Strict right-Noetherianity gives
\[
 \Delta_f:=-f(r_1)>0.
\]
Fix any $0<\varepsilon<\Delta_f$ and put
\[
 g=f+\varepsilon.
\]
All positive-order derivatives of $g$ coincide with
those of $f$, and
\[
 g(r_1)=f(r_1)+\varepsilon<0.
\]
Since $g$ is strictly increasing on $(r_1,\infty)$,
its largest real root is strictly greater than $r_1$.
The remaining derivative-root inequalities are unchanged.
Thus $g$ is strictly right-Noetherian.

The top integral identity in the hypothesis gives
\[
 \int_X\Phi_n^g(\alpha,\chi)
 =
 \int_X\Phi_n^f(\alpha,\chi)
       +\varepsilon\int_X\chi^n
 =
 \varepsilon\int_X\chi^n>0.
\]
For every $n-d\le q<n$, we also have
\[
 \Phi_q^g(\alpha,\chi)=\Phi_q^f(\alpha,\chi),
\]
because these forms do not involve the zeroth coefficient.
Therefore the same K\"ahler form $\kappa$ satisfies all
the mixed inequalities required for $g$.
Theorem~\ref{thm:positive}, proved above, yields a smooth
closed form
\[
 \omega_\varepsilon
 \in[\alpha]\cap\Ccal_{f+\varepsilon}(\chi).
\]
This proves\/ {\rm (ii)} for every
$0<\varepsilon<\Delta_f$.

If $d=1$, then $f+\varepsilon$ is a monic linear
polynomial for every $\varepsilon>0$.
The same top integral computation and the degree-one
case of Theorem~\ref{thm:positive} prove\/ {\rm (ii)}
for every $\varepsilon>0$.

Conversely, assume\/ {\rm (ii)}.
Fix one admissible $\varepsilon>0$ and a smooth closed
form
\[
 \omega\in[\alpha]\cap\Ccal_{f+\varepsilon}(\chi).
\]
For $n-d\le q<n$, put
\[
 \Theta_q=\Phi_q^f(\omega,\chi)
          =\Phi_q^{f+\varepsilon}(\omega,\chi).
\]
By Lemma~\ref{lem:geometry}, $\Theta_q$ is strictly
positive on every complex $q$-plane when $q\ge1$.
If $q=0$, then $\Theta_0=1$.

Fix an arbitrary K\"ahler form $\kappa$.
For every pair
\[
 n-d\le q\le\min\{p,n-1\},
\]
the form $\kappa^{p-q}\wedge\Theta_q$ is strictly
positive on every complex $p$-plane.
Indeed, in a $\kappa$-unitary coframe on such a plane,
its coefficient is a positive sum of the values of
$\Theta_q$ on complementary $q$-planes.

For $p\ge1$, consider the ratio
\[
 R_{p,q}(x,P)
 =
 \frac{
   \bigl(\kappa^{p-q}\wedge\Theta_q\bigr)|_P
 }{
   \bigl(\kappa^{p-q}\wedge\chi^q\bigr)|_P
 },
 \qquad
 P\in\operatorname{Gr}_p(T_x^{1,0}X).
\]
It is continuous and strictly positive.
The Grassmann bundle is compact, and there are only
finitely many relevant pairs $(p,q)$.
Consequently, there is a constant $\delta_\kappa>0$
such that
\[
 \kappa^{p-q}\wedge\Theta_q
 \ge
 \delta_\kappa\,\kappa^{p-q}\wedge\chi^q
\]
on every complex $p$-plane for every relevant pair.
When $(p,q)=(0,0)$ occurs, we additionally choose
$\delta_\kappa\le1$. 
Let $V\subset X$ be an irreducible analytic subvariety
of dimension $p$, including $V=X$.
Integration over its regular locus gives
\[
 \int_V\kappa^{p-q}\wedge\Theta_q
 \ge
 \delta_\kappa
 \int_V\kappa^{p-q}\wedge\chi^q>0.
\]
All the forms involved are closed, and $[\omega]=[\alpha]$.
Pairing with the closed integration current $[V]$
therefore gives
\[
 \int_V\kappa^{p-q}\wedge\Phi_q^f(\alpha,\chi)
 =
 \int_V\kappa^{p-q}\wedge\Theta_q.
\]
Hence
\[
 \int_V\kappa^{p-q}\wedge\Phi_q^f(\alpha,\chi)
 \ge
 \delta_\kappa
 \int_V\kappa^{p-q}\wedge\chi^q>0.
\]
This proves\/ {\rm (i)} and the asserted uniform
lower bound for every K\"ahler form $\kappa$.
\end{proof}

\begin{corollary}\label{cor:variable}
Suppose $c_i$ is constant for $i\ge1$, while $c_0$ is smooth.
For $d\ge2$, assume that
$f_x(t)=t^d-\sum_{i=1}^{d-1}\binom di c_it^i-c_0(x)$ is strictly-Noetherian at any point. For $d=1$, impose no further
root condition. Assume also the top identity
\[
 \int_X\alpha^d\chi^{n-d}
 -\sum_{i=1}^{d-1}\binom di c_i\int_X\alpha^i\chi^{n-i}
 =\int_Xc_0\chi^n.
\]
If the strict mixed inequalities hold, then \eqref{eq:properuniform} holds for some $\varepsilon_0>0$.
\end{corollary}
\begin{proof}
For $d\ge2$, write $P(t)=t^d-\sum_{i=1}^{d-1}\binom di c_it^i$ and
$r_1=r(P')$. The strict first gap says $c_0(x)>P(r_1)$.
We average $c_0$ against $\chi^n$ and obtain a constant
$\bar c_0>P(r_1)$. The averaged polynomial has the required root
ordering and top identity. Its positive-order derivatives are
unchanged, so Theorem~\ref{thm:main} gives \eqref{eq:properuniform}. For $d=1$, we have $\Phi_{n-1}^f=\chi^{n-1}$,
so \eqref{eq:properuniform} holds with $\varepsilon_0=1$.
\end{proof}

In full degree, we use the averaged coefficient to construct a
subsolution and then solve the equation with the original coefficient.

\subsection{The original full-degree equation}

\begin{proof}[Proof of Theorem~\ref{thm:fullPDE}]
We first assume (i). Put
\[
 P(t)=t^n-\sum_{i=1}^{n-1}\binom ni c_it^i,
 \qquad \bar c_0=\frac{\int_Xc_0\chi^n}{\int_X\chi^n}.
\]
We put $r_1=r(P')$ and
\[
 \delta_0=\min_{x\in X}\{c_0(x)-P(r_1)\}>0.
\]
The strict first root gap and compactness give
\begin{equation}\label{eq:averagedgap}
 \bar c_0-P(r_1)
 =\frac{1}{\int_X\chi^n}
       \int_X(c_0-P(r_1))\chi^n\ge\delta_0.
\end{equation}
Thus $\bar f=P-\bar c_0$ has the required root ordering, and
\eqref{eq:fullidentity} gives $\int_X\Phi_n^{\bar f}(\alpha,\chi)=0$.

For $0<\varepsilon<\delta_0$, Theorem~\ref{thm:main} gives a
smooth closed form $\xi\in\alpha\cap\Ccal_{\bar f+\varepsilon}(\chi)$.
Since
\[
 \partial_I\Pol_n(f_x)
 =\partial_I\Pol_n(\bar f+\varepsilon)
 \qquad(|I|\ge1),
\]
the derivative cones coincide. Hence $\xi$ is a subsolution
for the original equation \eqref{eq:fullPDE}.

We translate by $c=c_{n-1}$. As shown in
Section~\ref{sec:cones}, the polynomial $f_x(y+c)$ has vanishing
coefficient of $y^{n-1}$, constant positive-order coefficients,
and a smooth zeroth coefficient. The form $\xi-c\chi$ is a
positive definite subsolution, and the translated class satisfies
the integral identity. We apply \cite[Theorem 1.4]{Lin} and
translate back to obtain a unique smooth solution $\omega\in\alpha\cap\mathcal D_{f_x}(\chi)$.
The normalized derivative inequalities give $\omega-c\chi>0$.

Conversely, a solution in $\mathcal D_{f_x}(\chi)$ has strictly positive forms
$\Phi_q^{f_x}(\omega,\chi)$ for $0\le q<n$. Wedging with $\kappa^{p-q}$ and integrating
gives all the mixed inequalities. As in the proof of
Theorem~\ref{thm:main}, we take a positive minimum on the compact
Grassmann bundles to obtain a uniform lower bound.
\end{proof}

\section{Intersection inequalities on singular subvarieties}

We use $\Gamma_k$ as defined in the introduction. Recall that
$\Ccal_{t^k}(\chi)=\Gamma_k(\chi)$.

\begin{lemma}\label{lem:KT}
Suppose $\alpha\in\Kcal_{k,\chi}$. Let $V$ be irreducible of dimension $n-k\le p\le n$, and put
$s=k-n+p$. Then
\[
 u_i=\int_V\alpha^i\wedge\chi^{p-i}>0\quad(0\le i\le s),
 \qquad u_i^2\ge u_{i-1}u_{i+1}\quad(1\le i<s).
\]
For a point, we use $u_0=1$.
\end{lemma}

\begin{proof}
We first prove positivity of the numbers $u_i$. We choose a smooth
closed $k$-positive form $\omega\in\alpha$.
By compactness, we can choose $\epsilon>0$ so that
$\omega-\epsilon\chi$ remains $k$-positive. We can also choose
$M\ge1$ with $-M\chi\le\omega\le M\chi$.

The derivative forms of $t^k$ are positive. Expanding
$\omega=(\omega-\epsilon\chi)+\epsilon\chi$, we obtain
\begin{equation}\label{eq:formlower}
 \omega^a\chi^b\ge\epsilon^a\chi^{a+b},
 \qquad a\ge0,\quad b\ge n-k,\quad a+b\le n,
\end{equation}
in the cone of strongly positive forms. For $a=k$, the assertion
uses the positive top form in the definition of $\Gamma_k$.
We apply \eqref{eq:formlower} with $a=i$ and $b=p-i$ and integrate
on $V$. It follows that
\begin{equation}\label{eq:momentlower}
 u_i\ge\epsilon^i\int_V\chi^p>0.
\end{equation}
This proves the lemma when $s\le1$, including $k=1$.

For $s\ge2$, we choose a compact K\"ahler resolution
$\pi:W\to V\subset X$ and a K\"ahler form $\rho$ on $W$.
Write $A=\pi^*\omega$ and $B=\pi^*\chi$. The form $B$ is
semipositive and $-MB\le A\le MB$. We shall use the estimate
\begin{equation}\label{eq:absolutebound}
 -M^aB^{a+b}\rho^c\le A^aB^b\rho^c
 \le M^aB^{a+b}\rho^c,\qquad a+b+c=p.
\end{equation}
When $B$ is positive, we diagonalize $A$ relative to $B$.
Each coefficient is an average of products of $a$ eigenvalues,
each bounded in absolute value by $M$. This gives
\eqref{eq:absolutebound}. For semipositive $B$, we first replace
$B$ by $B+z\rho$, then let $z$ decrease to zero. The inequalities
pass to the limit in the closed cone of strongly positive forms.

We next show that one sufficiently small $\tau>0$ makes
\begin{equation}\label{eq:resperturb}
 A_\delta=A+\delta\rho
 \quad\text{$s$-positive relative to}\quad
 B_\delta=B+\tau\delta\rho
\end{equation}
for any $\delta>0$. Put $\eta=\tau\delta$.
For $1\le j\le s$, expand $A_\delta^jB_\delta^{p-j}$ and
group terms by the total exponent $r$ of $\rho$:
\begin{equation}\label{eq:groups}
 G_{j,r}=\rho^r
 \sum_{l=\max(0,r-p+j)}^{\min(j,r)}
 \binom jl\binom{p-j}{r-l}
 \delta^l\eta^{r-l}A^{j-l}B^{p-j-r+l}.
\end{equation}
If $r\le j$, the term $l=r$ has the lower bound
\[
 \binom jr\delta^r\epsilon^{j-r}B^{p-r}\rho^r,
\]
because $p-j\ge p-s=n-k$ and we can pull back
\eqref{eq:formlower}. any other term has $l<r$.
By \eqref{eq:absolutebound}, its negative part is bounded by
a constant times $\delta^r\tau^{r-l}B^{p-r}\rho^r$.
Thus, for some finite constant $C_{j,r}$ and $0<\tau\le1$,
\[
 G_{j,r}\ge\delta^r
 \left(\binom jr\epsilon^{j-r}-C_{j,r}\tau\right)
 B^{p-r}\rho^r.
\]

If $r>j$, we use the term $l=j$ instead. Its coefficient is
$\binom{p-j}{r-j}\delta^j\eta^{r-j}$.
Each other term gains the factor $\tau^{j-l}$ after division
by $\delta^j\eta^{r-j}$. Consequently, for a finite
$C'_{j,r}$,
\[
 G_{j,r}\ge\delta^j\eta^{r-j}
 \left(\binom{p-j}{r-j}-C'_{j,r}\tau\right)
 B^{p-r}\rho^r.
\]
We choose one $\tau>0$ so that all the finitely many parentheses
are positive. any group is then nonnegative. The group with
$r=p$ consists of the single term
$\delta^j\eta^{p-j}\rho^p>0$. This proves
\eqref{eq:resperturb}, even on the exceptional locus.

The classes $[A_\delta]$ and $[B_\delta]$ are $s$-positive
relative to the K\"ahler form $B_\delta$. We may therefore apply
\cite[Theorem B and Proposition 3.14]{Xiao} on $W$ to obtain
log concavity of
\[
 u_i(\delta)=\int_WA_\delta^iB_\delta^{p-i},
 \qquad 0\le i\le s.
\]
Keeping $\tau$ fixed, we let $\delta$ decrease to zero. These
integrals are polynomials in $\delta$, so their limits are
$\int_WA^iB^{p-i}=u_i$. Passing to the limit gives log concavity of $u_i$, while
\eqref{eq:momentlower} gives strict positivity.
\end{proof}

\begin{lemma}\label{lem:quotientmixed}
Suppose $\alpha\in\Kcal_{k,\chi}$ and $1\le\ell<k\le n$. Put
\[
 V_j=\int_X\alpha^j\chi^{n-j},\qquad C=V_k/V_\ell.
\]
Both volumes are positive. We use the forms $\Psi_q$ from \eqref{eq:psi}.
If \eqref{eq:quotendpoint} holds
for any irreducible proper analytic $V$, then
\[
 \int_V\chi^{p-q}\Psi_q>0,
 \qquad n-k\le q\le\min\{p,n-1\},
\]
for all irreducible analytic $V$, including $X$.
\end{lemma}

\begin{proof}
Fix $V$ of dimension $p$, and use the numbers $u_i$ of
Lemma~\ref{lem:KT}. Put $s=k-n+p$, $h=k-\ell$, and
$a=k-n+q$. If $q<n-\ell$, the required expression is
$u_a>0$. Otherwise $h\le a\le s$ and
\begin{equation}\label{eq:quotratio}
 \int_V\chi^{p-q}\Psi_q
 =u_a-C\theta_qu_{a-h}.
\end{equation}
Let $r_i=u_i/u_{i-1}$. Log concavity gives
\[
 r_1\ge r_2\ge\cdots\ge r_s>0.
\]
Since $a\le s$, each factor on the left below is at least
its corresponding factor on the right:
\[
 \frac{u_a}{u_{a-h}}=\prod_{i=a-h+1}^ar_i
 \ge\prod_{i=s-h+1}^sr_i=\frac{u_s}{u_{s-h}}.
\]
The definition of $\theta_q$ and $a=k-n+q$ give
\[
 \frac{\theta_{q+1}}{\theta_q}
       =\frac{a+1}{a+1-h}>1\qquad(q\ge n-\ell),
 \qquad \theta_n=1.
\]
If $p<n$, hypothesis \eqref{eq:quotendpoint}
says $u_s/u_{s-h}>C\theta_p$. Because $q\le p$, we obtain
$u_a/u_{a-h}>C\theta_q$, which makes
\eqref{eq:quotratio} positive.

If $V=X$, then $s=k$ and $u_s/u_{s-h}=C$. In this case
$q<n$, so $\theta_q<1$. We therefore have
$u_a/u_{a-h}\ge C>C\theta_q$. This proves the strict mixed
inequalities on $X$ as well.
\end{proof}

\section{The Hessian equations}\label{sec:pde}

We apply the geometric theorem to $k$-Hessian and Hessian quotient
equations.

\begin{proof}[Proof of Theorem~\ref{thm:hessian}]
We only need to prove that (i) implies (ii).
We apply Theorem~\ref{thm:positive} to $g(t)=t^k$.
Its derivative roots all equal zero and its top integral is $V_k>0$.
For any relevant $q$,
\[
 \Phi_q^g(\alpha,\chi)=\alpha^{k-n+q}\chi^{n-k}.
\]
By (i), all these mixed integrals are positive. Theorem~\ref{thm:positive}
therefore gives a smooth closed form in
$\alpha\cap\Ccal_{t^k}(\chi)=\alpha\cap\Gamma_k(\chi)$, proving (ii).

\end{proof}

Next, we consider the Hessian quotient equation. 

\begin{proof}
Set
$
f(t)=t^k-Ct^\ell.$
For $0\le j\le \ell$, the largest root of $f^{(j)}$ is
\[
r_j=
\left(
\frac{\ell!(k-j)!}{k!(\ell-j)!}C
\right)^{1/(k-\ell)}.
\]
Moreover,
\[
r_0>r_1>\cdots>r_\ell>0
=r_{\ell+1}=\cdots=r_{k-1}.
\]
Thus $f$ satisfies the root conditions of
Theorem~\ref{thm:main}.

Since $
\Phi_q^f(\alpha,\chi)=\Psi_q,$
Theorem~\ref{thm:main} yields, for sufficiently small
$\varepsilon>0$, a smooth closed form
$
\omega\in\alpha\cap\Ccal_{f+\varepsilon}(\chi).
$
The pointwise cone conditions imply
$
\omega\in\Gamma_k(\chi)
$
and
\begin{equation}\label{eq:subsolution}
k\omega^{k-1}\wedge\chi^{n-k}
-C\ell\omega^{\ell-1}\wedge\chi^{n-\ell}>0.
\end{equation}

Hence \cite[Corollary 3 and Proposition 22]{Szekelyhidi}
gives (ii).

 The necessity of (i) follows immediately from the pointwise
positivity of the corresponding $\Psi_q$ for an admissible
quotient solution.  Finally, if $\alpha\in\Kcal_{k,\chi}$,
Lemma~\ref{lem:quotientmixed} shows that
\eqref{eq:quotendpoint} implies all the mixed inequalities with
$\kappa=\chi$. The converse follows by taking $q=p$.
\end{proof}

For the quotient equation, there exists an $\epsilon_Q>0$
such that, for any irreducible proper $Z$ of dimension $p$,
\[
 \begin{cases}
 \displaystyle\int_Z\Psi_p\ge\epsilon_Q\int_Z\chi^p,
 & n-\ell\le p<n,\\[6pt]
 \displaystyle\int_Z\alpha^{k-n+p}\chi^{n-k}
       \ge\epsilon_Q\int_Z\chi^p,
 & n-k\le p<n-\ell.
 \end{cases}
\]

We test the quotient endpoint inequality only on proper
subvarieties. Indeed, at dimension $p=n$ its left side would be
$V_k-CV_\ell=0$ by the definition of $C$. Thus we interpret the subvariety condition in
\cite[Conjecture 1.5(2)]{Murakami} with $p<n$.

\section{Connected components and numerical paths}

Murakami  \cite[Conjecture 1.5(1)]{Murakami} formulated the path criterion for $k$-positive classes. He proved the case with Calabi symmetry. Applying the numerical criterion of Theorem~\ref{thm:positive}, we can solve Murakami's conjecture.

Recall the cone $\Kcal_{k,\chi}$ of $\Gamma_k$ classes. We define
\begin{align*}
 \Pcal_{k,\chi}
 &=\left\{\alpha:\int_Z\alpha^{k-n+p}\chi^{n-k}>0
       \text{ for any }Z,\ n-k+1\le p=\dim Z\le n\right\}.
\end{align*}
Here $Z$ ranges over all  irreducible  analytic subvarieties.

\begin{theorem}\label{thm:component}
 $\Kcal_{k,\chi}$ is the connected component of $\Pcal_{k,\chi}$ containing $[\chi]$. In particular, $\Kcal_{k,\chi}$ is open and convex.
Moreover, for any K\"ahler form $\kappa$, a class $\alpha\in\Pcal_{k,\chi}$ belongs to $\Kcal_{k,\chi}$ if and only if
\begin{equation}\label{eq:numericalpath}
\int_Z (\alpha+t[\kappa])^{k-n+p} \chi^{\,n-k}>0.
\end{equation}
for every $t\ge0$ and every irreducible analytic subvariety
$Z\subset M$ of dimension $n-k+1\le p\le n$.
\end{theorem}

\begin{proof}
Derivative positivity and positivity of the top Hessian form give
$\Kcal_{k,\chi}\subset\Pcal_{k,\chi}$.
Since $\Gamma_k$ is convex, 
$\Kcal_{k,\chi}$ is convex. To prove openness, choose smooth closed
representatives of a basis of $H^{1,1}(X,\R)$. A sufficiently
small linear combination of these forms preserves $k$-positivity,
uniformly over compact $X$. any K\"ahler class lies in
$\Kcal_{k,\chi}$.

We next prove that $\Kcal_{k,\chi}$ is closed relative to $\Pcal_{k,\chi}$. Suppose
$\alpha_j\in\Kcal_{k,\chi}$ and
$\alpha_j\to\alpha\in\Pcal_{k,\chi}$.
Fix an irreducible analytic $Z$ of dimension $p\ge n-k+1$,
put $s=k-n+p\ge1$, and write
\[
 u_r(\eta)=\int_Z\eta^r\chi^{p-r},\qquad 0\le r\le s.
\]
Lemma~\ref{lem:KT} gives a positive log-concave sequence at
each $\alpha_j$. Taking logarithms and interpolating between
the two endpoints yields
\begin{equation}\label{eq:interpolate}
 u_r(\alpha_j)\ge
 u_0^{1-r/s}u_s(\alpha_j)^{r/s},
 \qquad 0\le r\le s.
\end{equation}
For $s=1$ this is just the two endpoint identities.
We pass to the limit in \eqref{eq:interpolate}. Since
$u_s(\alpha)>0$ by membership in $\Pcal_{k,\chi}$, we obtain
\[
 u_r(\alpha)\ge u_0^{1-r/s}u_s(\alpha)^{r/s}>0.
\]
The argument applies to each fixed $Z$.

For a mixed index $n-k\le q\le\min\{p,n-1\}$, put
$r=k-n+q$. Then
\[
 \int_Z\chi^{p-q}\alpha^{k-n+q}\chi^{n-k}=u_r(\alpha)>0.
\]
If $p=n-k$, only the automatic constant test occurs.
If $p<n-k$, there is no test. The top condition on $X$ gives
$\int_X\alpha^k\chi^{n-k}>0$. Hence all hypotheses of
Theorem~\ref{thm:positive} for $g=t^k$ hold with
$\kappa=\chi$. We conclude that $\alpha\in\Kcal_{k,\chi}$,
which proves relative closedness.

It follows that $\Kcal_{k,\chi}$ is both open and closed
relative to $\Pcal_{k,\chi}$. It is connected by convexity and
contains $[\chi]$. Therefore it is exactly the connected
component containing $[\chi]$. A path starting in this component
stays there. Conversely, a class in $\Kcal_{k,\chi}$ joins
$[\chi]$ by a line segment in the convex cone. This proves the path criterion.

If $\alpha\in\Kcal_{k,\chi}$, we choose
$\omega\in\alpha\cap\Gamma_k(\chi)$. Then
$\omega+t\kappa\in\Gamma_k(\chi)$ for any $t\ge0$.
Thus \eqref{eq:numericalpath} holds for any $t\ge0$.
Conversely, suppose \eqref{eq:numericalpath} holds.
For sufficiently large $T$, the class $\alpha+T[\kappa]$
is K\"ahler, since $\theta+T\kappa>0$ for any fixed smooth
closed $\theta\in\alpha$ and sufficiently large $T$.
The segment $\{\alpha+t[\kappa]:0\le t\le T\}$ lies in
$\Pcal_{k,\chi}$. We join $\alpha+T[\kappa]$ to
$[\chi]$ inside the K\"ahler cone. The path criterion therefore
gives $\alpha$ in $\Kcal_{k,\chi}$.
\end{proof}

\begin{corollary}[Numerical branch criteria for the equations]
\label{cor:paths}
Let $\Pcal^0_{k,\chi}$ be the component of
$\Pcal_{k,\chi}$ containing $[\chi]$.
\begin{enumerate}
\item If $V_k>0$, then the pure Hessian equation
\eqref{eq:purePDE} has a unique $k$-positive solution for any
normalized smooth $F$ if and only if
$\alpha\in\Pcal^0_{k,\chi}$.
\item If $1\le\ell<k\le n$, $V_\ell\ne0$, and
$C=V_k/V_\ell>0$, then the quotient equation
\eqref{eq:quotPDE} has a unique $k$-positive solution if and only if
\[
 \alpha\in\Pcal^0_{k,\chi}
 \quad\text{and}\quad
 \int_Z\Psi_p>0\quad(n-\ell\le p=\dim Z<n).
\]
\end{enumerate}
In both assertions, component membership can be replaced by the
path condition or by requiring
$\alpha+t[\kappa]\in\Pcal_{k,\chi}$ for any $t\ge0$,
as in Theorem~\ref{thm:component}.
\end{corollary}
\begin{proof}
Theorem~\ref{thm:component} identifies
$\Pcal^0_{k,\chi}$ with $\Kcal_{k,\chi}$. For a class in this
cone, Theorem~\ref{thm:hessian} gives the pure Hessian equation,
and Theorem~\ref{thm:quotient} gives the quotient equation under
the proper endpoint inequalities. Conversely, a solution lies
in $\Gamma_k(\chi)$, so its class belongs to $\Pcal^0_{k,\chi}$.
In the quotient case, the solution also gives the proper
endpoint inequalities. 
\end{proof}

\section{Numerical criteria for the critical LYZ equation}
\label{sec:critical}

We apply Theorem~\ref{thm:main} to the polynomial of the critical LYZ equation. It has degree $n-1$ and its derivative cone is the critical subsolution cone.

\subsection{The polynomial and its derivatives}

We let
\begin{equation}\label{eq:criticalpoly}
 f_{\mathrm{cr}}(t)=\frac1n\operatorname{Im}(t+\sqrt{-1})^n.
\end{equation}
For $0\le j\le n-2$,
we have
\begin{equation}\label{eq:criticalrootsderivative}
 f_{\mathrm{cr}}^{(j)}(t)
 =\frac{(n-1)!}{(n-j)!}\operatorname{Im}(t+\sqrt{-1})^{n-j}.
\end{equation}

For $m\ge2$, let $\theta=\mathrm{arccot}{(t)}$,
then
\(
 \operatorname{Im}(t+\sqrt{-1})^m
       =\frac{\sin(m\theta)}{\sin^m\theta}.
\)
We  obtain all its real roots:
\[
 t_a=\cot\frac{a\pi}{m},\qquad 1\le a\le m-1.
\]
The largest root is
$\cot\frac{\pi}{m}$. Equation~\eqref{eq:criticalrootsderivative} gives
\begin{equation}\label{eq:criticalrootchain}
 r\bigl(f_{\mathrm{cr}}^{(j)}\bigr)
       =\cot\frac\pi{n-j},\qquad 0\le j\le n-2.
\end{equation}
Thus
$f_{\mathrm{cr}}$ is strongly strictly right-Noetherian. In particular, it
satisfies the  hypothesis of Theorem~\ref{thm:main}.

For this polynomial, all the forms have the same expression:
\begin{equation}\label{eq:criticalPhi}
 \Phi_q^{f_{\mathrm{cr}}}(A,B)
 =\frac1q\operatorname{Im}(A+\sqrt{-1}B)^q,
 \qquad 1\le q\le n.
\end{equation}
Indeed, differentiating the top form $n-q$ times and normalizing gives
\[
 \left.\frac{d^{n-q}}{ds^{n-q}}\right|_{s=0}
 \Phi_n^{f_{\mathrm{cr}}}(A+sS,B)
 =\frac{(n-1)!}{(q-1)!}S^{n-q}\Phi_q^{f_{\mathrm{cr}}}(A,B).
\]

For $q=1$, the imaginary part is $B$. Thus the mixed inequalities for $q=1$ are automatic, and only
$q\ge2$ appears in Theorem~\ref{thm:critical}.

We also have
\begin{equation}\label{eq:criticalPhizero}
 Z_{\alpha,\chi}\in\R_{<0}, \qquad
 \int_X\Phi_n^{f_{\mathrm{cr}}}(\alpha,\chi)
       =\frac1n\operatorname{Im}Z_{\alpha,\chi}=0.
\end{equation}
The negative sign also fixes the principal argument to be $\pi$.

\subsection{The derivative components and the phase}

We use $\operatorname{arccot}t=\pi/2-\arctan t\in(0,\pi)$.
For a real eigenvalue vector $\lambda$, set
\[
 \theta_i=\operatorname{arccot}\lambda_i,
 \qquad \lambda_i+\sqrt{-1}
       =\sqrt{1+\lambda_i^2}\,e^{\sqrt{-1}\theta_i}.
\]
Polarization of \eqref{eq:criticalpoly} gives
\begin{align}
 F(\lambda)&=\Pol_n(f_{\mathrm{cr}})(\lambda)
       =\frac1n\operatorname{Im}\prod_{i=1}^n(\lambda_i+\sqrt{-1}),
  \label{eq:criticalpolar}\\
 \partial_jF(\lambda)
       &=\frac1n\operatorname{Im}\prod_{i\ne j}(\lambda_i+\sqrt{-1}).
     \label{eq:criticalpartial}
\end{align}
In particular,
\[
 \partial_jF(\lambda)
 =\frac1n\prod_{i\ne j}\sqrt{1+\lambda_i^2}\,
                       \sin\left(\sum_{i\ne j}\theta_i\right).
\]
The large positive diagonal lies in the region
$0<\sum_{i\ne j}\theta_i<\pi$.
The set
\[
 \left\{(\theta_i)_{i\ne j}:\theta_i>0,
                         \sum_{i\ne j}\theta_i<\pi\right\}
\]
is connected. The
distinguished positive component of $\partial_jF$ is exactly
$
 \sum_{i\ne j}\theta_i<\pi.$

We obtain
\begin{equation}\label{eq:criticalcomponent}
 \lambda\in\mathcal D_{f_{\mathrm{cr}}}
 \quad\Longleftrightarrow\quad
 \sum_{i\ne j}\arctan\lambda_i>(n-3)\frac\pi2
                         \quad\text{for any}j.
\end{equation}
where $\mathcal D_f$ denotes the distinguished derivative cone.
Hence $
 \mathcal D_{f_{\mathrm{cr}}+\varepsilon}
       =\mathcal D_{f_{\mathrm{cr}}}.$
Thus any form in $\Ccal_{f_{\mathrm{cr}}+\varepsilon}(\chi)$
satisfies the critical subsolution inequality.

To prove positivity of the lower derivative forms, suppose \eqref{eq:criticalcomponent} holds and let
$I\subset\{1,\ldots,n\}$ have $1\le |I|=q<n$.
Choose $j\notin I$ and then 
\(
 0<\sum_{i\in I}\theta_i
       \le\sum_{i\ne j}\theta_i<\pi.
\)
Thus we have
\begin{equation}\label{eq:criticalminorpositive}
 \operatorname{Im}\prod_{i\in I}(\lambda_i+\sqrt{-1})
 =\left(\prod_{i\in I}\sqrt{1+\lambda_i^2}\right)
                 \sin\left(\sum_{i\in I}\theta_i\right)>0.
\end{equation}
 Thus these forms are strictly strongly
positive for any $1\le q<n$.

\subsection{Proof of the numerical criterion}

\begin{proof}[Proof of Theorem~\ref{thm:critical}]
Assume (i). By \eqref{eq:criticalPhi}, each derivative form is a
positive multiple of the corresponding imaginary part. Hence
\[
 \int_V[\kappa]^{p-q}\Phi_q^{f_{\mathrm{cr}}}(\alpha,\chi)>0,
 \qquad 1\le q\le\min\{p,n-1\}.
\]
The case $q=1$ is automatic, since
\[
 \int_V[\kappa]^{p-1}\Phi_1^{f_{\mathrm{cr}}}(\alpha,\chi)
       =\int_V[\kappa]^{p-1}\chi>0.
\]
Thus all the mixed inequalities hold.
The polynomial is strictly right-Noetherian by
\eqref{eq:criticalrootchain}, and its top integral vanishes by
\eqref{eq:criticalPhizero}. We apply Theorem~\ref{thm:main} to
obtain a smooth closed form $\omega\in\alpha$ such that
\[
 \omega\in\Ccal_{f_{\mathrm{cr}}+\varepsilon}(\chi)
             \subset\mathcal D_{f_{\mathrm{cr}}}(\chi)
\]
for small $\varepsilon>0$. By
\eqref{eq:criticalcomponent}, this is  (ii).

Assume (ii), and choose a smooth closed $\theta\in\alpha$.
The $\partial\bar\partial$ lemma gives $A=\theta+\ddc\underline u$.
Since $\chi$ is K\"ahler, $\arg Z_{\alpha,\chi}=\pi$, and
$\underline u$ satisfies \eqref{eq:criticalsub}, we may apply
\cite[Theorem 1.1]{FYZ}. Then we obtain a smooth solution of
\eqref{eq:criticalPDE}, this proves the first assertion of (iii).
By \eqref{eq:criticalminorpositive},
each form
\[
 \Omega_q=\operatorname{Im}(A+\sqrt{-1}\chi)^q,
                    \qquad 1\le q<n,
\]
is strictly strongly positive. Fix any K\"ahler form $\kappa$.
For each  $(p,q)$, the continuous function
\[
 (x,P)\longmapsto
 \frac{(\kappa^{p-q}\Omega_q)|_P}
      {(\kappa^{p-q}\chi^q)|_P},
 \qquad P\in\Gr_p(T_x^{1,0}X),
\]
has a positive minimum. We take the minimum over the finitely
many pairs $(p,q)$ and call it $\delta_\kappa$.
Integration over the regular part of $V$
gives \eqref{eq:criticaluniform}. Closedness and Stokes' theorem
replace $A$ with its class $\alpha$. Hence (iii) holds with one
constant independent of $V,p,q$.  Finally, (iii) implies (i)
because its right side is positive.

\end{proof}

\section{Examples concerning the choice of component}
\label{sec:scope}

For $f(t)=t^n-a$, where $a>0$, we have
\[
 \Phi_q^f(\alpha,\chi)=\alpha^q,
 \qquad 0\le q<n.
\]
Under $\int_X\alpha^n=a\int_X\chi^n$, Theorem~\ref{thm:main}
therefore characterizes K\"ahler classes by the mixed inequalities.
These select the K\"ahler component of the numerical locus in
\cite[Theorem 0.1]{DP}; Theorem~\ref{thm:component} recovers that
component description when $k=n$.
For example, on a surface, $\alpha=-\chi$ satisfies
\[
 \int_X\alpha^2=\int_X\chi^2>0,
 \qquad \int_X[\kappa]\alpha=-\int_X[\kappa]\chi<0.
\]
Thus the mixed inequality on $X$ excludes the negative class despite
its positive top self-intersection.

A K\"ahler class can also lie on a different polynomial branch.
In dimension three, take
\[
 f(t)=(t-1)(t-2)(t-3),\qquad \alpha=\chi.
\]
This polynomial is strongly strictly right-Noetherian, and
$f(1)=0$ gives $\Phi_n^f(\chi,\chi)=0$. However,
\[
 \Phi_1^f(\alpha,\chi)=\alpha-2\chi=-\chi,
 \qquad
 \int_X[\kappa]^2\Phi_1^f(\alpha,\chi)<0.
\]
A solution in the distinguished derivative cone would make this
intersection number positive. Hence the K\"ahler assumption on
$\alpha$ does not by itself select the required branch.

\begingroup
\raggedright
\providecommand{\bysame}{\leavevmode\hbox to3em{\hrulefill}\thinspace}
\providecommand{\MR}{\relax\ifhmode\unskip\space\fi MR }
\providecommand{\MRhref}[2]{%
  \href{http://www.ams.org/mathscinet-getitem?mr=#1}{#2}
}
\providecommand{\href}[2]{#2}

\endgroup

\end{document}